\documentclass[11pt]{article}

\usepackage{amsfonts}
\usepackage{amscd}
\usepackage{amssymb}
\usepackage{amsthm}
\usepackage{amsmath} 
\usepackage{pdfsync}
\usepackage{enumerate}

\usepackage{pictexwd}

\theoremstyle{plain}
\newtheorem{thm}{Theorem}[section]

\newtheorem{prop}[thm]{Proposition}

\theoremstyle{definition}

\newtheorem{remark}[thm]{Remark}
\theoremstyle{example}

\theoremstyle{remark}

\numberwithin{equation}{section}

\def\cB{\mathcal{B}}

\def\cE{\mathcal{E}}

\def\cH{\mathcal{H}}

\def\cO{\mathcal{O}}

\def\cR{\mathcal{R}}

\def\cW{\mathcal{W}}

\def\CC{\mathbb{C}}

\def\QQ{\mathbb{Q}}

\def\ZZ{\mathbb{Z}}

\def\fg{\mathfrak{g}}

\def\fgl{\mathfrak{gl}}  
\def\fsl{\mathfrak{sl}}  
\def\fso{\mathfrak{so}}  
\def\fsp{\mathfrak{sp}}

\def\half{\hbox{$\frac12$}}

\def\<{\langle}
\def\>{\rangle}

\usepackage{color}
\definecolor{dred}{rgb}{.65, 0, 0.15}
\definecolor{zajj}{rgb}{0.4, 0, .6}
\definecolor{gray}{rgb}{0.6, .6, .6}

\makeatletter
\renewcommand{\@makefnmark}{\mbox{\textsuperscript{}}}
\makeatother

\title{Affine and degenerate affine BMW algebras:
The center}
\author{
Zajj Daugherty\\
Department of Mathematics, \\Statistics, and Computer Science\\
St.\ Olaf College\\
Northfield, Minnesota 55057 USA\\
daugherz@stolaf.edu \\ 
\and
Arun Ram \\
Department of Mathematics and Statistics \\
University of Melbourne \\
Parkville VIC 3010 Australia \\
aram@unimelb.edu.au \\
\and
Rahbar Virk \\
Department of Mathematics\\
University of California, Davis\\
One Shields Ave\\
Davis, CA 95616\\
virk@math.ucdavis.edu\\
}
\date{}
\newcommand{\comment}[1]{}
\usepackage[pdftex,bookmarks]{hyperref}

\begin{document}

\maketitle

\begin{abstract} 
The degenerate affine and affine BMW algebras arise naturally in the 
context of Schur-Weyl duality for orthogonal and symplectic Lie algebras and quantum 
groups, respectively.  Cyclotomic BMW algebras, affine Hecke algebras, cyclotomic Hecke algebras, and their degenerate versions are quotients. In this paper the theory is unified by 
treating the orthogonal and symplectic cases simultaneously; we make an exact parallel
between the degenerate affine and affine cases  via a new algebra which takes the role of the affine braid group for the degenerate setting.   A main result of this paper is an identification of the centers of the affine and degenerate affine BMW algebras in terms of rings of symmetric functions which satisfy a ``cancellation property'' or ``wheel condition'' (in the degenerate case, a reformulation of a result of Nazarov). Miraculously, these same rings also arise in Schubert calculus, as the cohomology and K-theory of isotropic Grassmanians and symplectic loop Grassmanians. We also establish new intertwiner-like identities which, when projected to the center, produce the recursions for central elements given previously by Nazarov for degenerate affine BMW algebras, and by Beliakova-Blanchet for affine BMW algebras. 

\smallskip


\end{abstract}
\setcounter{tocdepth}{2}
\tableofcontents

\section{Introduction}

The degenerate affine BMW algebras $\cW_k$ and the affine BMW algebras
$W_k$ arise naturally in the context of Schur-Weyl duality and the application
of Schur functors to modules in category $\cO$ for orthogonal and symplectic
Lie algebras and quantum groups (using the Schur functors
of \cite{Ze}, \cite{AS}, and \cite{OR}).   The degenerate algebras $\cW_k$
were introduced in \cite{Naz} and the affine versions $W_k$ appeared in \cite{OR},
following foundational work of \cite{Ha1}-\cite{Ha3}.  The representation theory of
$\cW_k$ and $W_k$ contains the representation theory of any quotient: in particular, 
the degenerate cyclotomic BMW algebras $\cW_{r,k}$, the cyclotomic BMW algebras $W_{r,k}$, the degenerate affine Hecke algebras $\cH_k$, the affine Hecke algebras $H_k$, the degenerate cyclotomic Hecke algebras $\cH_{r,k}$, and the cyclotomic Hecke algebras $H_{r,k}$ as quotients. The representation theory of the affine BMW algebras is an image of the representation theory of category $\cO$ for orthogonal and symplectic Lie algebras and their
quantum groups in the same way that the affine Hecke algebras arise in Schur-Weyl duality
with the enveloping algebra of $\fgl_n$ and its Drinfeld-Jimbo quantum group. 

In the literature, the algebras $\cW_k$ and $W_k$ have often been treated separately.
One of the goals of this paper is to unify the theory.  To do this we have begun
by adjusting the definitions of the algebras carefully to make the presentations
match, relation by relation.  In the same way that the affine BMW algebra is 
a quotient of the group algebra of the affine braid group, we have defined a
new algebra, the degenerate affine braid algebra which has the degenerate
affine BMW algebra and the degenerate affine Hecke algebras as quotients.
We have done this carefully, to ensure that 
the Schur-Weyl duality framework is completely analogous for both the degenerate affine 
and the affine cases.  We have also
added a parameter $\epsilon$ (which takes values $\pm1$) so that both the orthogonal
and symplectic cases can be treated simultaneously.  Our new presentations of the algebras
$\cW_k$ and $W_k$ are given in section \ref{sec:defns}.

In section \ref{sec:identities} we consider some remarkable recursions for generating central elements in 
the algebras $\cW_k$ and $W_k$.  These recursions were given by Nazarov \cite{Naz}
in the degenerate case, and then extended to the affine BMW algebra by Beliakova-Blanchet \cite{BB}.
Another proof in the affine cyclotomic case appears in \cite[Lemma 4.21]{RX2} 
and, in the degenerate case, in \cite[Lemma 4.15]{AMR}.
In all of these proofs, the recursion is obtained by a rather mysterious and tedious computation.
We show that there is an ``intertwiner'' like identity in the full algebra which, when
``projected to the center" produces the Nazarov recursions.  Our approach dramatically 
simplifies the proof and provides insight into where these recursions are coming from. 
Moreover, the proof is exactly analogous
in both the degenerate and the affine cases, and includes the parameter $\epsilon$, so that both the
orthogonal and symplectic cases are treated simultaneously.

In section \ref{sec:center} we identify the center of the degenerate and affine BMW algebras.  In the degenerate
case this has been done in \cite{Naz}.  Nazarov stated that the center of the degenerate affine 
BMW algebra is the subring of the ring of symmetric functions generated by the odd power 
sums.   We identify the ring in a different way, as the subring of symmetric functions with the Q-cancellation property, in the language of Pragacz \cite{Pr}.  This is a fascinating ring.
Pragacz identifies it as the cohomology ring of orthogonal and symplectic Grassmannians;
the same ring appears again as the cohomology of the loop Grassmannian for the symplectic 
group in \cite{LSS, La}; and references for the relationship of this ring to the projective representation theory of the symmetric group, the BKP hierarchy of differential equations, representations of Lie superalgebras, and twisted Gelfand pairs are found in 
\cite[Ch.\ II \S 8]{Mac}.  
For the affine BMW algebra, the Q-cancellation property can be generalized well to provide a suitable description of the center.  From our perspective, one would expect that the
ring which appears as the center of the affine BMW algebra
should also appear as the K-theory of the orthogonal and symplectic Grassmannians and
as the K-theory of the loop Grassmannian for the symplectic group, but we are not aware
that these identifications have yet been made in the literature.

This paper is part of a more comprehensive work on affine and degenerate affine BMW 
algebras.  In future work [DRV] we may:
\begin{enumerate}
\item[(a)] set up the commuting actions between the algebras $\cW_k$ and $W_k$ and the enveloping algebras of orthogonal and symplectic Lie algebras and their quantum groups, 
\item[(b)] show how the central elements which arise in the Nazarov recursions coincide with
central elements studied in Baumann \cite{Bau},
\item[(c)] provide a new approach to admissibility conditions by providing ``universal admissible parameters'' in an appropriate ground ring (arising naturally, from Schur-Weyl duality, as the center of the enveloping algebra, or quantum group),
\item[(d)] classify and construct the irreducible representations of $\cW_k$ and $W_k$ by multisegments, and
\item[(e)] define Khovanov-Lauda-Rouquier analogues of the affine BMW algebras.
\end{enumerate}
Many parts of this program are already available in the works of Goodman, Rui, Wilcox-Yu, and others (see, for example, \cite{RS1}-\cite{RS2}, \cite{RX1}-\cite{RX2}, \cite{Go1}-\cite{Go3}, \cite{GH1}-\cite{GH3}, \cite{WY1}-\cite{WY2}, \cite{Yu}).   Some parts of our 
work are also available at \cite{Ra}.

\smallskip\noindent
\textbf{Acknowledgements:}  Significant work on this paper was done while the authors 
were in residence at the Mathematical Sciences Research Institute (MSRI) in 2008, and the 
writing was completed when A. Ram was in residence at the Hausdorff Insitute for Mathematics (HIM) in 2011.  We thank MSRI and HIM for hospitality, support,
and a wonderful working environment during these stays.  This research has 
been partially supported by the National Science Foundation (DMS-0353038) and the 
Australian Research Council (DP-0986774).  We thank S. Fomin for providing the reference \cite{Pr} and Fred Goodman for providing the reference \cite{BB}, many informative discussions, detailed proofreading, and for much help in processing the theory around admissibility conditions.
We thank J. Enyang for his helpful comments on the manuscript.

\section{Affine and degenerate affine BMW algebras}
\label{sec:defns}

In this section, we define the affine Birman-Murakami-Wenzl (BMW) algebra $W_k$ and its degenerate version $\cW_k$.  We have adjusted the definitions to unify the theory. In particular,
in section \ref{s:degenaffbdgp}, we define a new algebra, the degenerate affine braid 
algebra $\cB_k$, which has the degenerate affine BMW algebras $\cW_k$ and the degenerate affine Hecke algebras $\cH_k$  as quotients. The motivation for the definition of $\cB_k$ is that 
the affine BMW algebras $W_k$ and the affine Hecke algebras $H_k$ are quotients of the group algebra of affine braid group $CB_k$. 

The definition of the degenerate affine braid algebra $\cB_k$ also makes the Schur-Weyl duality framework completely analogous in both the affine and degenerate affine cases. Both $\cB_k$ and $CB_k$ are designed to act on tensor space of the form $M \otimes V^{\otimes k}$.  In the degenerate affine case this is an action commuting with a complex semisimple Lie algebra $\fg$,
and in the affine case this is an action commuting with the Drinfeld-Jimbo quantum group $U_q\fg$.  The degenerate affine and affine BMW algebras arise when $\fg$ is $\fso_n$ or $\fsp_n$
and $V$ is the first fundamental representation and the 
degenerate affine and affine Hecke algebras arise when $\fg$ is $\fgl_n$ or $f\sl_n$ and 
$V$ is the first fundamental representation. In the case when $M$ is the trivial representation and
$\fg$ is $\fso_n$, the ``Jucys-Murphy'' elements $y_1, \dots, y_k$ in $\cB_k$ become the
``Jucys-Murphy'' elements for the Brauer algebras used in \cite{Naz} and, in the case that 
$\fg = \fgl_n$, these become the classical Jucys-Murphy elements in the group algebra of the symmetric group.
The Schur-Weyl duality actions are explained in \cite{DRV} and \cite{Ra}.


\subsection{The degenerate affine braid algebra $\cB_k$}\label{s:degenaffbdgp}
Let $C$ be a commutative ring, and let $S_k$ denote the symmetric group on $\{1,\ldots, k\}$. For $i\in\{1,\ldots, k\}$, write $s_i$ for the transposition in $S_k$ that switches $i$ and $i+1$.
The \emph{degenerate affine braid algebra} is the algebra $\cB_k$ over $C$ generated by
\begin{equation}\label{gensA}
t_u\quad (u\in S_k),
\qquad
\kappa_0, \kappa_1,
\qquad\hbox{and}\qquad
y_1,\ldots, y_k,
\end{equation}
with relations
\begin{equation}\label{easyrels}
t_ut_v = t_{uv}, \qquad y_iy_j = y_jy_i,
\quad
\kappa_0\kappa_1=\kappa_1\kappa_0,
\quad
\kappa_0 y_i = y_i \kappa_0,
\quad
\kappa_1 y_i = y_i\kappa_1,
\end{equation}
\begin{equation}\label{kyrels}
\quad
\kappa_0 t_{s_i} = t_{s_i}\kappa_0,
\qquad 
\kappa_1t_{s_1}\kappa_1t_{s_1}
=t_{s_1}\kappa_1t_{s_1}\kappa_1,
\qquad\hbox{and}\qquad
\kappa_1t_{s_j} = t_{s_j}\kappa_1,\ \hbox{for $j\ne 1$,}
\end{equation}
\begin{equation}\label{rel:graded_braid3}
t_{s_i}(y_i + y_{i+1}) = (y_i + y_{i+1})t_{s_i}, 
\quad\hbox{and}\qquad
y_jt_{s_i} = t_{s_i} y_j, \quad \hbox{for  $j \neq i, i+1$,}
\end{equation}
and
\begin{equation}
\label{rel:graded_braid5}
t_{s_i} t_{s_{i+1}}\gamma_{i,i+1}t_{s_{i+1}} t_{s_i} 
= \gamma_{i+1,i+2},
\qquad\hbox{where}\quad
\gamma_{i,i+1} =  y_{i+1}-t_{s_i}y_it_{s_i}   \text{ for  $i=1,\ldots, k-2$.}
\end{equation}

In the degenerate affine braid algebra $\cB_k$ let $c_0 = \kappa_0$ and
\begin{equation}\label{kappadefn}
c_j= \kappa_0+ 2(y_1+\ldots+y_j),
\qquad\hbox{so that}\quad
y_j = \hbox{$\frac{1}{2}$}(c_j - c_{j-1}),
\quad \hbox{for } j=1, \dots, k.
\end{equation}
Then $c_0,\ldots, c_k$ commute with each other, commute with $\kappa_1$, and the relations \eqref{rel:graded_braid3} are equivalent to 
\begin{equation}\label{kapparelns}
	t_{s_i}c_j = c_j t_{s_i},
	\qquad\hbox{for $j\ne i$.}
\end{equation}

\begin{thm}
The degenerate affine braid algebra $\cB_k$ has another presentation by generators
\begin{equation}\label{gensB}
t_u,\ \hbox{for $u\in S_k$,} \qquad
\kappa_0,\ldots, \kappa_k
\qquad\hbox{and}\qquad
\gamma_{i,j},\ \ \hbox{for $0\le i,j\le k$ with $i\ne j$,}
\end{equation}
and relations
\begin{equation}\label{tildeBrels1}
t_ut_v = t_{uv}, \qquad 
t_w \kappa_i t_{w^{-1}} = \kappa_{w(i)},
\qquad
t_w \gamma_{i,j} t_{w^{-1}} = \gamma_{w(i), w(j)},
\end{equation}
\begin{equation}\label{rel:graded_braid4}
\kappa_i\kappa_j = \kappa_j\kappa_i,
\qquad \kappa_i \gamma_{\ell,m} = \gamma_{\ell,m}\kappa_i,
\end{equation}
\begin{equation}\label{rel:graded_braid6}
\gamma_{i,j}=\gamma_{j,i},\qquad
\gamma_{p,r}\gamma_{\ell,m} = \gamma_{\ell,m}\gamma_{p,r},
\qquad\hbox{and}\qquad
\gamma_{i,j}(\gamma_{i,r}+\gamma_{j,r}) = (\gamma_{i,r}+\gamma_{j,r})\gamma_{i,j},
\end{equation}
for $p\ne \ell$ and $p\ne m$ and $r\ne \ell$ and $r\ne m$
and $i\ne j$, $i\ne r$ and $j\ne r$.
\end{thm}

The commutation relations between the $\kappa_i$ and the
$\gamma_{i,j}$ can be rewritten in the form
\begin{equation}
[\kappa_r, \gamma_{\ell,m}]=0,
\qquad
[\gamma_{i,j}, \gamma_{\ell,m}] = 0,
\qquad\hbox{and}\qquad
[\gamma_{i,j}, \gamma_{i,m}] = [\gamma_{i,m}, \gamma_{j,m}],
\end{equation}
for all $r$ and all $i\ne \ell$ and $i\ne m$ and $j\ne \ell$ and $j\ne m$.


\begin{proof}
The generators in \eqref{gensB} are written in terms of the generators in
\eqref{gensA}
by the formulas
\begin{equation}
\kappa_0=\kappa_0,
\qquad
\kappa_1=\kappa_1,
\qquad
t_w=t_w,
\end{equation}
\begin{equation}
\gamma_{0,1} = y_1 - \hbox{$\frac12$}\kappa_1,
\quad\hbox{and}\quad
\gamma_{j,j+1} = y_{j+1} - t_{s_j}y_jt_{s_j},
\ \ \hbox{for $j=1,\ldots, k-1$,}
\end{equation}
and 
\begin{equation}\label{ktdefs}
\kappa_m = t_u \kappa_1 t_{u^{-1}}, 
\qquad
\gamma_{0,m} = t_u \gamma_{0,1}t_{u^{-1}}
\qquad\hbox{and}\qquad
\gamma_{i,j} = t_v\gamma_{1,2}t_{v^{-1}},
\end{equation}
for $u,v\in S_k$ such that $u(1)=m$, $v(1)=i$ and 
$v(2) = j$.

The generators in \eqref{gensA} are written in terms of the generators in
\eqref{gensB} by the formulas
\begin{equation}\label{gensAtogensB}
\kappa_0=\kappa_0,
\qquad
\kappa_1=\kappa_1,
\qquad
t_w=t_w,
\qquad\hbox{and}\qquad
y_j=\hbox{$\frac12$}\kappa_j + \sum_{0\le \ell<j} \gamma_{\ell, j}.
\end{equation}

\smallskip\noindent
Let us show that relations in (2.2-5) follow from the relations in (2.9-2.11).
\begin{enumerate}
\item[(a)] The relation $t_ut_v= t_{uv}$ in \eqref{easyrels} is the first relation in \eqref{tildeBrels1}.
\item[(b)] The relation $y_iy_j=y_jy_i$ in \eqref{easyrels}:  Assume that $i<j$. Using the relations in 
\eqref{rel:graded_braid4} and \eqref{rel:graded_braid6},
\begin{align*}
[y_i,y_j]
&= \big[ \hbox{$\frac12$}\kappa_i + \sum_{\ell<i} \gamma_{\ell,i},
\hbox{$\frac12$}\kappa_j + \sum_{m<j} \gamma_{m,j}\big]
= \big[ \sum_{\ell<i} \gamma_{\ell,i},
\sum_{m<j} \gamma_{m,j}\big] \\
&= \sum_{\ell<i} \big[\gamma_{\ell,i},
\sum_{m<j} \gamma_{m,j}\big] 
= \sum_{\ell<i} \big[\gamma_{\ell,i},
(\gamma_{\ell,j}+\gamma_{i,j})+\sum_{m<j\atop m\ne \ell, m\ne i} \gamma_{m,j}\big] = 0.
\end{align*}
\item[(c)] The relation $\kappa_0\kappa_1 = \kappa_1\kappa_0$ in \eqref{easyrels} 
is part of the first relation in \eqref{rel:graded_braid4}, 
and the relations $\kappa_0y_i = y_i \kappa_0$ and $\kappa_1 y_i = y_i \kappa_1$ in \eqref{easyrels} follow from the relations
$\kappa_i\kappa_j=\kappa_j\kappa_i$ and $\kappa_i \gamma_{\ell,m} = \gamma_{\ell,m}\kappa_i$
in \eqref{rel:graded_braid4}.
\item[(d)] The relations $\kappa_0t_{s_i} = t_{s_i} \kappa_0$  and 
$\kappa_1t_{s_j} = t_{s_j}\kappa_1$ for $j\ne 1$ from \eqref{kyrels} follow from the relation
$t_w \kappa_i t_w^{-1} = \kappa_{w(i)}$ in \eqref{tildeBrels1}, and the relation
$\kappa_1t_{s_1}\kappa_1t_{s_1}=t_{s_1}\kappa_1t_{s_1}\kappa_2$ from \eqref{kyrels}
follows from $\kappa_1\kappa_2=\kappa_2\kappa_1$,
which is part of the first relation in \eqref{rel:graded_braid4}.
\item[(e)] The relations in \eqref{rel:graded_braid3} and \eqref{rel:graded_braid5} all 
follow from the relations
$t_w\kappa_i t_{w^{-1}} = \kappa_{w(i)}$ and $t_w \gamma_{i,j} t_{w^{-1}} = \gamma_{w(i), w(j)}$
in \eqref{tildeBrels1}.
\end{enumerate}
\noindent
To complete the proof let us show that the relations of (2.9-11) follow from the relations in (2.2-5).
\begin{enumerate}
\item[(a)] The relation $t_ut_v= t_{uv}$ in \eqref{tildeBrels1} is the first relation in \eqref{easyrels}.
\item[(b)]  The relations $t_w\kappa_i t_{w^{-1}}=\kappa_{w(i)}$ in \eqref{tildeBrels1} follow from 
the first and last relations in \eqref{kyrels} (and force the definition of $\kappa_m$ in \eqref{ktdefs}).
\item[(c)]  Since $\gamma_{0,1}=y_1-\hbox{$\frac12$}\kappa_1$, the relations
$t_w \gamma_{0,j} t_{w^{-1}} = \gamma_{0, w(j)}$ in \eqref{rel:graded_braid4} follow from the
last relation in each of \eqref{kyrels} and \eqref{rel:graded_braid3} (and force
the definition of $\gamma_{0,m}$ in \eqref{ktdefs}). 
\item[(d)] Since $\gamma_{1,2} = y_2-t_{s_1}y_1t_{s_1}$, the first relation in \eqref{rel:graded_braid3}
gives
\begin{equation}\label{stable}
t_{s_1}\gamma_{1,2}t_{s_1} - \gamma_{1,2} = (t_{s_1}y_2t_{s_1}-y_1)-y_2+t_{s_1}y_1t_{s_1}
=t_{s_1}(y_1+y_2)t_{s_1}-(y_1+y_2)=0.
\end{equation}
The relations
$t_w \gamma_{1,2} t_{w^{-1}} = \gamma_{w(1),w(2)}$ in \eqref{tildeBrels1} then follow from
\eqref{stable} and 
the last relation in \eqref{rel:graded_braid3} (and force the definitions $\gamma_{i,j}=t_v\gamma_{1,2}t_{v^{-1}}$
in \eqref{ktdefs}).
\item[(e)]  The third relation in \eqref{easyrels} is $\kappa_0\kappa_1=\kappa_1\kappa_0$
and the second relation in \eqref{kyrels} gives $\kappa_1\kappa_2=\kappa_2\kappa_1$.
The relations $\kappa_i\kappa_j=\kappa_j\kappa_i$ in \eqref{rel:graded_braid4} then
follow from the second set of relations in \eqref{tildeBrels1}.

\item[(f)]  The second relation in \eqref{kyrels} gives $[\kappa_1,\kappa_2]=0$.  Using this
and the relations in \eqref{easyrels},
\begin{equation}\label{kappasumt}
[\kappa_1, \gamma_{0,2}+\gamma_{1,2}] = [\kappa_1, (y_2-\hbox{$\frac12$}\kappa_2-\gamma_{1,2})+\gamma_{1,2}]
=[\kappa_1,\hbox{$\frac12$}\kappa_2]=0,
\end{equation}
and
\begin{equation}\label{twitht}
[\gamma_{0,1}, \gamma_{0,2}+\gamma_{1,2}]
=[y_1 - \hbox{$\frac12$}\kappa_1, y_2 - \hbox{$\frac12$}\kappa_2]
=\hbox{$\frac14$}[\kappa_1,\kappa_2]=0,
\end{equation}
so that
$$[\gamma_{0,1},\kappa_2] = [\gamma_{0,1}, 2y_2 - 2(\gamma_{0,2}+\gamma_{1,2})]
=[\gamma_{0,1}, 2y_2] = [y_1 - \hbox{$\frac12$}\kappa_1, 2y_2] = -[\kappa_1, y_2] = 0.$$
Conjugating the last relation by $t_{s_1}$ gives 
$$[\kappa_1, \gamma_{0,2}]=0,\qquad\hbox{and thus}\qquad[\kappa_1,\gamma_{1,2}]=0,$$
by \eqref{kappasumt}.
By the third and fourth relations in \eqref{easyrels},
$$[\kappa_0, \gamma_{0,1}] = [\kappa_0, y_1-\hbox{$\frac12$}\kappa_1]=0,
\qquad\hbox{and}\qquad
[\kappa_1, \gamma_{0,1}]=[\kappa_1, y_1-\hbox{$\frac12$}\kappa_1]=0.$$
By the relations in \eqref{kyrels} and \eqref{easyrels},
$$[\kappa_0, \gamma_{1,2}] = [\kappa_0, y_2 - t_{s_1}y_1t_{s_1}]=0
\qquad\hbox{and}\qquad
[\kappa_1, \gamma_{2,3}]=[\kappa_1, y_3-t_{s_2}y_2t_{s_2}]=0.$$
Putting these together with the (already established) relations in \eqref{tildeBrels1}
provides the second set of relations in \eqref{rel:graded_braid4}.

\item[(g)]  From the commutativity of the $y_i$ and the second relation in 
\eqref{rel:graded_braid3}
$$\gamma_{1,2}\gamma_{3,4}
= (y_2-t_{s_1}y_1t_{s_1})(y_4 - t_{s_3}y_3t_{s_3})
= (y_4 - t_{s_3}y_3t_{s_3})(y_2-t_{s_1}y_1t_{s_1})
= \gamma_{3,4}\gamma_{1,2}.
$$
By the last relation in \eqref{easyrels} and the last relation in \eqref{kyrels},
$$[\gamma_{0,1},\gamma_{2,3}] = [y_1-\hbox{$\frac12$}\kappa_1,y_3-t_{s_2}y_2t_{s_2}]=0.$$
Together with the  (already established) relations in \eqref{tildeBrels1}, we obtain the
first set of relations in \eqref{rel:graded_braid6}.
\item[(h)]  Conjugating \eqref{twitht} by $t_{s_2}t_{s_1}t_{s_2}$ gives $[\gamma_{0,2}, \gamma_{0,3}+\gamma_{2,3}]=0$, and this and the (already established) relations in \eqref{rel:graded_braid4} and
the first set of relations in \eqref{rel:graded_braid6} provide
\begin{align*}
0&=[y_2,y_3]=[\hbox{$\frac12$}\kappa_2+\gamma_{0,2}+\gamma_{1,2},
\hbox{$\frac12$}\kappa_3+\gamma_{0,3}+\gamma_{1,3}+\gamma_{2,3}] \\
&=[\gamma_{0,2}+\gamma_{1,2}, \gamma_{0,3}+\gamma_{1,3}+\gamma_{2,3}] 
=[\gamma_{1,2}, \gamma_{0,3}+\gamma_{1,3}+\gamma_{2,3}]
=[\gamma_{1,2},\gamma_{1,3}+\gamma_{2,3}].
\end{align*}
Note also that 
\begin{align*}
[\gamma_{1,2}, \gamma_{1,0}+\gamma_{2,0}]
&=[\gamma_{1,2}, \gamma_{0,1}+\gamma_{0,2}]
=-[\gamma_{0,1},\gamma_{1,2}]+[\gamma_{1,2},\gamma_{0,2}] \\
&=[\gamma_{0,1},\gamma_{0,2}]+[\gamma_{1,2},\gamma_{0,2}]
=t_{s_1}[\gamma_{0,2}+\gamma_{1,2}, \gamma_{0,1}]t_{s_1} = 0,
\end{align*}
by (two applications of) \eqref{twitht}.
The last set of relations in \eqref{rel:graded_braid6} now follow from the last set of relations in \eqref{tildeBrels1}.
\end{enumerate}
\end{proof}
%
\smallskip\noindent
By the first formula in \eqref{kappadefn} and the last formula in \eqref{gensAtogensB},
\begin{equation}\label{cingensB}
c_j = \sum_{i=0}^j \kappa_i + 2\sum_{0\le \ell<m\le j} \gamma_{\ell,m}.
\end{equation}


\subsection{The degenerate affine BMW algebra $\cW_k$}\label{s:degenaff}

Let $C$ be a commutative ring and let $\cB_k$ be the degenerate affine braid algebra
over $C$ as defined in Section \ref{s:degenaffbdgp}.
Define $e_i$ in the degenerate affine braid algebra by 
\begin{equation}\label{rel:e_defn} 
t_{s_i}y_i = y_{i+1}t_{s_i} -(1-e_i),
\qquad \hbox{for $i=1,2,\ldots, k-1$,}
\end{equation}
so that, with $\gamma_{i,i+1}$ as in \eqref{rel:graded_braid5},
\begin{equation}\label{eq:tt-e_defn}\gamma_{i,i+1}t_{s_i} = 1 - e_i.\end{equation}

Fix constants
$$\epsilon=\pm1\qquad\hbox{and}\qquad
z_0^{(\ell)}\in C,\quad\hbox{for $\ell\in \ZZ_{\ge0}$.}
$$
The \emph{degenerate affine Birman-\!Wenzl-Murakami (BMW) algebra}  $\cW_k$ 
(with parameters $\epsilon$ and $z_0^{(\ell)}$) is the quotient of the degenerate affine 
braid algebra $\cB_k$ by the relations
\begin{equation}\label{rel:untwisting} 
e_i t_{s_i} = t_{s_i} e_i = \epsilon e_i,
\qquad
e_i t_{s_{i-1}}e_i = e_i t_{s_{i+1}}e_i = \epsilon e_i,
\end{equation}
\begin{equation}\label{rel:unwrapping} 
e_1 y_1^{\ell} e_1 = z_0^{(\ell)} e_1,
\qquad
e_i(y_i + y_{i+1}) = 0 = (y_i + y_{i+1})e_i.
\end{equation}
  Conjugating \eqref{rel:e_defn} by $t_{s_i}$ and using the first relation in \eqref{rel:untwisting} gives
\begin{equation} \label{altedefn}
y_it_{s_i} = t_{s_i}y_{i+1} - (1-e_i).
\end{equation}
Then, by \eqref{eq:tt-e_defn} and \eqref{rel:graded_braid5},
\begin{equation}\label{skein}
\gamma_{i,i+1} = t_{s_i} - \epsilon e_i,
\qquad\hbox{and}\qquad
e_{i+1} = t_{s_i}t_{s_{i+1}}e_it_{s_{i+1}}t_{s_i}.
\end{equation}
Multiply the second relation in \eqref{skein} on the left and the right by $e_i$,
and then use the relations in \eqref{rel:untwisting} to get
$$
e_ie_{i+1}e_i
= e_it_{s_i}t_{s_{i+1}}e_it_{s_{i+1}}t_{s_i}e_i 
= e_it_{s_{i+1}}e_it_{s_{i+1}}e_i
=\epsilon~\! e_it_{s_{i+1}}e_i = e_i,
$$
so that
\begin{equation}\label{rel:skein}
e_ie_{i\pm1}e_i = e_i.
\qquad\qquad\hbox{Note that}\qquad
e_i^2 = z_0^{(0)} e_i
\end{equation}
is a special case of the first identity in \eqref{rel:unwrapping}.  
The relations
\begin{equation}\label{tangle1}
e_{i+1} e_i = e_{i+1}t_{s_i}t_{s_{i+1}}, \quad e_i e_{i+1} = t_{s_{i+1}}t_{s_i} e_{i+1},
\end{equation}
\begin{equation}\label{tangle2}
t_{s_i} e_{i+1} e_i = t_{s_{i+1}} e_i,  \quad \text{ and } \quad
e_{i+1} e_i t_{s_{i+1}} = e_{i+1} t_{s_i}
\end{equation}
result from
\begin{align*}
e_{i+1}t_{s_i}t_{s_{i+1}}
&=\epsilon~\! e_{i+1}t_{s_i}e_{i+1}t_{s_i}t_{s_{i+1}}
=e_{i+1}t_{s_{i+1}}t_{s_i}e_{i+1}t_{s_i}t_{s_{i+1}}=e_{i+1}e_i, \\
t_{s_{i+1}}t_{s_i}e_{i+1}
&= \epsilon~\!\! t_{s_{i+1}}t_{s_i}e_{i+1}t_{s_i}e_{i+1}
=t_{s_{i+1}}t_{s_i}e_{i+1}t_{s_i}t_{s_{i+1}}e_{i+1}
=e_ie_{i+1}, \\
t_{s_i}e_{i+1}e_i &=\epsilon~\!  t_{s_i}e_{i+1}t_{s_i}e_i=\epsilon~\! t_{s_{i+1}}e_it_{s_{i+1}}e_i=t_{s_{i+1}}e_i,
\quad\hbox{and}
\\
e_{i+1}e_it_{s_{i+1}}
&= \epsilon~\!e_{i+1}t_{s_{i+1}}e_it_{s_{i+1}}
=\epsilon~\!e_{i+1}t_{s_i}e_{i+1}t_{s_i}
=e_{i+1}t_{s_i}.
\end{align*}


\begin{remark}  A consequence (see \eqref{genfnadm}) of the defining relations of $\cW_k$ 
is the equation
\begin{equation*}
\left( z_0(-u) - \big( \hbox{$\frac12$} + \epsilon\,u \big) \right)
\left( z_0(u) - \big( \hbox{$\frac12$} - \epsilon\,u \big) \right) e_1
=
\big( \hbox{$\frac12$} - \epsilon\,u \big)\big( \hbox{$\frac12$} + \epsilon\,u \big) e_1,
\end{equation*}
where $z_0(u)$ is the generating function
$$z_0(u) = \sum_{\ell\in \ZZ_{\ge 0}} z_0^{(\ell)}u^{-\ell}.$$
This means that, unless the parameters $z_0^{(\ell)}$ are chosen carefully, it is likely
that $e_1=0$ in $\cW_k$.
\end{remark}

\begin{remark} 
From the point of view of the Schur-Weyl duality for the degenerate affine BMW algebra 
(see \cite{AS} and \cite{Ra}) the 
natural choice of base ring is the center of the enveloping algebra of the
orthogonal or symplectic Lie algebra which, by the Harish-Chandra isomorphism, is
isomorphic to the subring of symmetric functions given by
\begin{equation*}
\label{dBMW base ring}
C = \{ z\in \CC[h_1,\ldots, h_r]^{S_r}\ |\ z(h_1,\ldots, h_r) = z(-h_1,h_2,\ldots, h_r)\},
\end{equation*}
where the symmetric group $S_r$ acts by permuting the variables $h_1,\ldots, h_r$.
Here the constants $z_0^{(\ell)}\in C$ are given, explicitly, 
by setting the generating function 
$$
\hbox{$z_0(u)$
equal, up to a normalization, to}\quad
\prod_{i=1}^r \frac{ (u+\frac12+h_i)(u+\frac12-h_i)}{(u-\frac12 - h_i)(u-\frac12+h_i)}.
$$
This choice of $C$ and the $z_0^{(\ell)}$ are the \emph{universal admissible parameters} for
$\cW_k$.  This point of view will be explained in \cite{DRV}.
\end{remark}

\begin{remark} Careful manipulation of the defining relations of $\cW_k$ provides
an inductive presentation of $\cW_k$ as 
$$\cW_k = \cW_{k-1}e_{k-1}\cW_{k-1} 
+\cW_{k-1}t_{s_{k-1}}\cW_{k-1}
+\sum_{\ell\in \ZZ_{\ge0}} \cW_{k-1}y_k^\ell \cW_{k-1},
$$
and provides that 
$$
e_k\cW_ke_k = \cW_{k-1}e_k,
\qquad\hbox{and}\qquad
\begin{matrix}
\cW_k &\longrightarrow &\cW_{k-1}e_k \\
b &\longmapsto &e_kbe_k
\end{matrix}
$$
is a $(\cW_{k-1}, \cW_{k-1})$-bimodule homomorphism.
These structural facts are important to the understanding of $\cW_k$ by 
``Jones basic constructions''.
Under the conditions of Theorem \ref{thm:W_basis}(a) it is true,
but not immediate from the defining relations,
that the natural homomorphism $\cW_{k-1}\to \cW_k$ is injective so that 
$\cW_{k-1}$ is a subalgebra of $\cW_k$.  These useful structural results for
the algebras $\cW_k$ are justified in \cite{AMR}.
\end{remark}

\subsubsection{Quotients of $\cW_k$}

The \emph{degenerate affine Hecke algebra} $\cH_k$ is the quotient of $\cW_k$ by the relations
\begin{equation}\label{rel:gradedhecke}
e_i=0, 
\qquad\hbox{for $i=1,\ldots, k-1$.}
\end{equation}
Fix $b_1,\ldots, b_r\in C$.  The \emph{degenerate cyclotomic BMW algebra} 
$\cW_{r,k}(b_1,\ldots, b_r)$ 
is the degenerate affine BMW algebra with the additional relation 
\begin{equation}\label{cycrelation}
(y_1-b_1)\cdots (y_1-b_r) = 0.
\end{equation}
The \emph{degenerate cyclotomic Hecke algebra} $\cH_{r,k}(b_1,\ldots, b_r)$ is the
degenerate affine Hecke algebra $\cH_k$ with the additional relation
\eqref{cycrelation}.

\begin{remark}\label{polysubalg}
Since the composite map
$C[y_1, \ldots, y_k] \to \cB_k \to \cW_k\to \cH_k$ is injective (see \cite[Theorem 3.2.2]{Kl}) and the 
last two maps are surjections, it follows that the polynomial ring 
$C[y_1,\ldots, y_k] $ is a subalgebra of $\cB_k$ and $\cW_k$.
\end{remark}

\begin{remark}  A consequence of the relation \eqref{cycrelation} in $\cW_{r,k}(b_1,\ldots, b_r)$ 
is 
\begin{equation}
\left(z_0(u) + u-\hbox{$\frac12$}\right)e_1
= (u-\hbox{$\frac12$}(-1)^r)\left(\prod_{i=1}^r \frac{u+b_i}{u-b_i}\right)e_1.
\end{equation}
This equation makes the data of the values $b_i$ almost equivalent to the 
data of the $z_0^{(\ell)}$.
\end{remark}


\subsection{The affine braid group $B_k$}\label{s:affbdgp}

The \emph{affine braid group} $B_k$ is the group given by
generators $T_1, T_2, \ldots, T_{k-1}$ and $X^{\varepsilon_1}$, 
with relations
\begin{align}
T_i T_j &= T_j T_i, &\mbox{if } j \neq i \pm 1 \label{affrel1},\\
T_i T_{i+1} T_i &= T_{i+1} T_i T_{i+1}, &\mbox{for } i=1, 2,\ldots, k-2,\label{affrel2}\\
X^{\varepsilon_1}T_1 X^{\varepsilon_1} T_1 
&= T_1 X^{\varepsilon_1} T_1 X^{\varepsilon_1}, \label{affrel3}\\
X^{\varepsilon_1}T_i &= T_i X^{\varepsilon_1},
&\mbox{for } i=2,3,\ldots, k-1. \label{affrel4}
\end{align}
The affine braid group is isomorphic to the group of braids in the thickened annulus,
where the generators $T_i$ and $X^{\varepsilon_1}$ are identified with the diagrams
\begin{equation}
T_i = 
\beginpicture
\setcoordinatesystem units <.5cm,.5cm>         
\setplotarea x from -5 to 3.5, y from -2 to 2    
\put{${}^i$} at 0 1.2      %
\put{${}^{i+1}$} at 1 1.2      %
\put{$\bullet$} at -3 .75      %
\put{$\bullet$} at -2 .75      %
\put{$\bullet$} at -1 .75      %
\put{$\bullet$} at  0 .75      
\put{$\bullet$} at  1 .75      %
\put{$\bullet$} at  2 .75      %
\put{$\bullet$} at  3 .75      %
\put{$\bullet$} at -3 -.75          %
\put{$\bullet$} at -2 -.75          %
\put{$\bullet$} at -1 -.75          %
\put{$\bullet$} at  0 -.75          
\put{$\bullet$} at  1 -.75          %
\put{$\bullet$} at  2 -.75          %
\put{$\bullet$} at  3 -.75          %
\plot -4.5 1.25 -4.5 -1.25 /
\plot -4.25 1.25 -4.25 -1.25 /
\ellipticalarc axes ratio 1:1 360 degrees from -4.5 1.25 center 
at -4.375 1.25
\put{$*$} at -4.375 1.25  
\ellipticalarc axes ratio 1:1 180 degrees from -4.5 -1.25 center 
at -4.375 -1.25 
\plot -3 .75  -3 -.75 /
\plot -2 .75  -2 -.75 /
\plot -1 .75  -1 -.75 /
\plot  2 .75   2 -.75 /
\plot  3 .75   3 -.75 /
\setquadratic
\plot  0 -.75  .05 -.45  .4 -0.1 /
\plot  .6 0.1  .95 0.45  1 .75 /
\plot 0 .75  .05 .45  .5 0  .95 -0.45  1 -.75 /
\endpicture
\qquad\hbox{and}\qquad
X^{\varepsilon_1} = 
~~\beginpicture
\setcoordinatesystem units <.5cm,.5cm>         
\setplotarea x from -5 to 3.5, y from -2 to 2    
\put{$\bullet$} at -3 0.75      %
\put{$\bullet$} at -2 0.75      %
\put{$\bullet$} at -1 0.75      %
\put{$\bullet$} at  0 0.75      
\put{$\bullet$} at  1 0.75      %
\put{$\bullet$} at  2 0.75      %
\put{$\bullet$} at  3 0.75      %
\put{$\bullet$} at -3 -0.75          %
\put{$\bullet$} at -2 -0.75          %
\put{$\bullet$} at -1 -0.75          %
\put{$\bullet$} at  0 -0.75          
\put{$\bullet$} at  1 -0.75          %
\put{$\bullet$} at  2 -0.75          %
\put{$\bullet$} at  3 -0.75          %
\plot -4.5 1.25 -4.5 -0.13 /
\plot -4.5 -0.37   -4.5 -1.25 /
\plot -4.25 1.25 -4.25  -0.13 /
\plot -4.25 -0.37 -4.25 -1.25 /
\ellipticalarc axes ratio 1:1 360 degrees from -4.5 1.25 center 
at -4.375 1.25
\put{$*$} at -4.375 1.25  
\ellipticalarc axes ratio 1:1 180 degrees from -4.5 -1.25 center 
at -4.375 -1.25 
\plot -2 0.75  -2 -0.75 /
\plot -1 0.75  -1 -0.75 /
\plot  0 0.75   0 -0.75 /
\plot  1 0.75   1 -0.75 /
\plot  2 0.75   2 -0.75 /
\plot  3 0.75   3 -0.75 /
\setlinear
\plot -3.3 0.25  -4.1 0.25 /
\ellipticalarc axes ratio 2:1 180 degrees from -4.65 0.25  center 
at -4.65 0 
\plot -4.65 -0.25  -3.3 -0.25 /
\setquadratic
\plot  -3.3 0.25  -3.05 .45  -3 0.75 /
\plot  -3.3 -0.25  -3.05 -0.45  -3 -0.75 /
\endpicture
.
\label{braidfig1}
\end{equation}
For $i=1,\ldots, k$ define
\begin{equation}\label{BraidMurphy}
X^{\varepsilon_i}=T_{i-1}T_{i-2}\cdots T_2T_1 
X^{\varepsilon_1}T_1T_2\cdots T_{i-2}T_{i-1} = 
~~\beginpicture
\setcoordinatesystem units <.5cm,.5cm>         
\setplotarea x from -5 to 3.5, y from -2 to 2    
\put{${}^i$} at 1 1.2 
\put{$\bullet$} at -3 0.75      %
\put{$\bullet$} at -2 0.75      %
\put{$\bullet$} at -1 0.75      %
\put{$\bullet$} at  0 0.75      
\put{$\bullet$} at  1 0.75      %
\put{$\bullet$} at  2 0.75      %
\put{$\bullet$} at  3 0.75      %
\put{$\bullet$} at -3 -0.75          %
\put{$\bullet$} at -2 -0.75          %
\put{$\bullet$} at -1 -0.75          %
\put{$\bullet$} at  0 -0.75          
\put{$\bullet$} at  1 -0.75          %
\put{$\bullet$} at  2 -0.75          %
\put{$\bullet$} at  3 -0.75          %
\plot -4.5 1.25 -4.5 -0.13 /
\plot -4.5 -0.37   -4.5 -1.25 /
\plot -4.25 1.25 -4.25  -0.13 /
\plot -4.25 -0.37 -4.25 -1.25 /
\ellipticalarc axes ratio 1:1 360 degrees from -4.5 1.25 center 
at -4.375 1.25
\put{$*$} at -4.375 1.25  
\ellipticalarc axes ratio 1:1 180 degrees from -4.5 -1.25 center 
at -4.375 -1.25 
\plot -3 0.75  -3 -0.1 /
\plot -2 0.75  -2 -0.1 /
\plot -1 0.75  -1 -0.1 /
\plot  0 0.75   0 -0.1 /
\plot -3 -.35  -3 -0.75 /
\plot -2 -.35   -2 -0.75 /
\plot -1 -.35   -1 -0.75 /
\plot  0 -.35    0 -0.75 /
\plot  2 0.75   2 -0.75 /
\plot  3 0.75   3 -0.75 /
\setlinear
\plot -3.2 0.25  -4.1 0.25 /
\plot -2.8 0.25  -2.2 0.25 /
\plot -1.8 0.25  -1.2 0.25 /
\plot -.8 0.25  -.2 0.25 /
\plot  .2 0.25  .5 0.25 /
\plot -3.3 -.25  .5 -.25 /
\ellipticalarc axes ratio 2:1 180 degrees from -4.65 0.25  center 
at -4.65 0 
\plot -4.65 -0.25  -3.3 -0.25 /
\setquadratic
\plot  .5 0.25  .9 .45  1 0.75 /
\plot  .5  -0.25  .9 -0.45 1 -0.75 /
\endpicture.
\end{equation}
The pictorial computation
$$
X^{\varepsilon_j}X^{\varepsilon_i}=
\beginpicture
\setcoordinatesystem units <.5cm,.5cm>         
\setplotarea x from -5.5 to 3.5, y from -2 to 2    
\put{\scriptsize $j$} at 1 2
\put{\scriptsize $i$} at -1 2 
\put{$\bullet$} at -3 1.5      %
\put{$\bullet$} at -2 1.5      %
\put{$\bullet$} at -1 1.5      %
\put{$\bullet$} at  0 1.5      
\put{$\bullet$} at  1 1.5      %
\put{$\bullet$} at  2 1.5      %
\put{$\bullet$} at  3 1.5      %
\put{$\bullet$} at -3 -1.6          %
\put{$\bullet$} at -2 -1.6          %
\put{$\bullet$} at -1 -1.6          %
\put{$\bullet$} at  0 -1.6          
\put{$\bullet$} at  1 -1.6          %
\put{$\bullet$} at  2 -1.6          %
\put{$\bullet$} at  3 -1.6          %
\plot -4.5 2 -4.5 0.45 /
\plot -4.25 2 -4.25  0.45 /
\plot -4.5 0.15 -4.5 -0.87 /
\plot -4.25 0.15 -4.25  -0.87 /
\plot -4.5 -1.18   -4.5 -2 /
\plot -4.25 -1.18 -4.25 -2 /
\ellipticalarc axes ratio 1:1 360 degrees from -4.5 2 center 
at -4.375 2
\put{$*$} at -4.375 2  
\ellipticalarc axes ratio 1:1 180 degrees from -4.5 -2 center 
at -4.375 -2 
\plot -3 1.5  -3 0.5 /
\plot -2 1.5  -2 0.5 /
\plot -1 1.5  -1 0.5 /
\plot  0 1.5   0 0.5 /
\plot -3 .1  -3 -.85 /
\plot -2 .1  -2 -.85 /
\plot -1 .1  -1 0 /
\plot  0 .1   0 -1.5 /
\plot  1 0   1 -1.5 /
\plot -3 -1.15  -3 -1.5 /
\plot -2 -1.15   -2 -1.5 /
\plot  2 1.5   2 -1.5 /
\plot  3 1.5   3 -1.5 /
\setlinear
\plot -3.2 1  -4.1 1 /
\plot -2.8 1  -2.2 1 /
\plot -1.8 1  -1.2 1 /
\plot -.8 1  -.2 1 /
\plot  .2 1  .5 1 /
\plot -3.3 .3  .5 .3 /
\ellipticalarc axes ratio 2:1 180 degrees from -4.65 1  center 
at -4.65 .65 
\plot -4.65 0.3  -3.3 0.3 /
\setquadratic
\plot  .5 1  .9 1.2  1 1.5 /
\plot  .5  0.3  .9 .2 1 0 /
\setlinear
\plot -3.2 -.5  -4.1 -.5 /
\plot -2.8 -.5  -2.2 -.5 /
\plot -1.8 -.5  -1.5 -.5 /
\plot -3.3 -1  -1.5 -1 /
\ellipticalarc axes ratio 1:1 180 degrees from -4.65 -.5  center 
at -4.65 -.75 
\plot -4.65 -1  -3.3 -1 /
\setquadratic
\plot  -1.5 -.5  -1.1 -.3  -1 0 /
\plot  -1.5  -1  -1.1 -1.2 -1 -1.5 /
\endpicture
=
\beginpicture
\setcoordinatesystem units <.5cm,.5cm>         
\setplotarea x from -5.5 to 3.5, y from -2 to 2    
\put{\scriptsize $j$} at 1 2
\put{\scriptsize $i$} at -1 2 
\put{$\bullet$} at -3 1.5      %
\put{$\bullet$} at -2 1.5      %
\put{$\bullet$} at -1 1.5      %
\put{$\bullet$} at  0 1.5      
\put{$\bullet$} at  1 1.5      %
\put{$\bullet$} at  2 1.5      %
\put{$\bullet$} at  3 1.5      %
\put{$\bullet$} at -3 -1.6          %
\put{$\bullet$} at -2 -1.6          %
\put{$\bullet$} at -1 -1.6          %
\put{$\bullet$} at  0 -1.6          
\put{$\bullet$} at  1 -1.6          %
\put{$\bullet$} at  2 -1.6          %
\put{$\bullet$} at  3 -1.6          %
\plot -4.5 2 -4.5 0.62 /
\plot -4.25 2 -4.25  0.62 /
\plot -4.5 0.35 -4.5 -0.87 /
\plot -4.25 0.35 -4.25  -0.87 /
\plot -4.5 -1.18   -4.5 -2 /
\plot -4.25 -1.18 -4.25 -2 /
\ellipticalarc axes ratio 1:1 360 degrees from -4.5 2 center 
at -4.375 2
\put{$*$} at -4.375 2  
\ellipticalarc axes ratio 1:1 180 degrees from -4.5 -2 center 
at -4.375 -2 
\plot -3 1.5  -3 0.65 /
\plot -3 .3  -3 -.85 /
\plot -3 -1.15  -3 -1.5 /
\plot -2 1.5  -2 0.65 /
\plot -2 .3  -2 -.85 /
\plot -2 -1.15   -2 -1.5 /
\plot -1 0  -1 -.85 /
\plot -1 -1.15   -1 -1.5 /
\plot  0 1.5   0 -.85 /
\plot 0 -1.15   0 -1.5 /
\plot  1 1.5   1 0 /
\plot  2 1.5   2 -1.5 /
\plot  3 1.5   3 -1.5 /
\setlinear
\plot -3.2 1  -4.1 1 /
\plot -2.8 1  -2.2 1 /
\plot -1.8 1  -1.5 1 /
\plot -3.3 .5  -1.5 .5 /
\ellipticalarc axes ratio 1:1 180 degrees from -4.65 1  center 
at -4.65 .75 
\plot -4.65 0.5  -3.3 0.5 /
\setquadratic
\plot  -1.5 1  -1.1 1.2  -1 1.5 /
\plot  -1.5  0.5  -1.1 .3 -1 0 /
\setlinear
\plot -3.2 -.3  -4.1 -.3 /
\plot -2.8 -.3  -2.2 -.3 /
\plot -1.8 -.3  -1.2 -.3 /
\plot -.8 -.3  -.2 -.3 /
\plot  .2 -.3  .5 -.3 /
\plot -3.3 -1  .5 -1 /
\ellipticalarc axes ratio 2:1 180 degrees from -4.65 -.3  center 
at -4.65 -.65 
\plot -4.65 -1  -3.3 -1 /
\setquadratic
\plot  .5 -.3  .9 -.15  1 0 /
\plot  .5  -1  .9 -1.2 1 -1.5 /
\endpicture
=X^{\varepsilon_i}X^{\varepsilon_j}
$$
shows that the $X^{\varepsilon_i}$ all commute with each other.


\subsection{The affine BMW algebra $W_k$}
\label{sec:W_k-defn}

Let $C$ be a commutative ring and let $CB_k$ be the group algebra of the affine
braid group.  Fix constants
$$q, z\in C
\qquad\hbox{and}\qquad
Z_0^{(\ell)}\in C,\quad\hbox{for $\ell\in \ZZ$,}
$$
with $q$ and $z$ invertible.  Let $Y_i = zX^{\varepsilon_i}$ so that
\begin{equation}
	\label{rel:Y}
		Y_1 = zX^{\varepsilon_1},\qquad
		Y_i = T_{i-1}Y_{i-1}T_{i-1},
	\qquad\hbox{and}\qquad
		Y_iY_j=Y_jY_i,\ \ \hbox{for $1\le i,j\le k$.}
\end{equation}
In the affine braid group
\begin{equation}
	\label{rel:Antisymmetry_T}
		T_i Y_iY_{i+1}  = Y_iY_{i+1}T_i.
\end{equation}
Assume that $q-q^{-1}$ is invertible in $C$ and define $E_i$ in the group algebra of the 
affine braid group by
\begin{equation}\label{rel:E_Defn} 
T_i Y_i = Y_{i+1} T_i -(q-q^{-1})Y_{i+1}(1-E_i).
\end{equation}
The \emph{affine BMW algebra} $W_k$ 
is the quotient of the group algebra $CB_k$ of the affine braid group $B_k$ by the relations
\begin{equation}\label{rel:Untwisting} 
E_i T^{\pm1}_i = T^{\pm1}_i E_i = z^{\mp1}E_i,
\qquad
E_i T_{i-1}^{\pm1}E_i = E_i T_{i+1}^{\pm1}E_i = z^{\pm1}E_i,
\end{equation}
\begin{equation}\label{rel:Unwrapping} 
E_1 Y_1^{\ell} E_1 = Z_0^{(\ell)} E_1,
\qquad
E_iY_iY_{i+1} = E_i = Y_iY_{i+1}E_i.
\end{equation}
Since $Y_{i+1}^{-1}(T_iY_i)Y_{i+1} = Y_{i+1}^{-1}Y_iY_{i+1}T_i=Y_iT_i$, conjugating 
\eqref{rel:E_Defn} by $Y_{i+1}^{-1}$ gives
\begin{equation}\label{altEdefn}
Y_i T_i = T_i Y_{i+1} - (q - q^{-1})(1 - E_i)Y_{i+1}.
\end{equation}
Left multiplying \eqref{rel:E_Defn} by $Y_{i+1}^{-1}$ and using the second identity in 
\eqref{rel:Y}
shows that \eqref{rel:E_Defn} is equivalent to
$T_i - T_i^{-1} = (q - q^{-1})(1 - E_i)$, so that
\begin{equation}\label{Skein} 
E_i = 1-\frac{T_i-T_i^{-1}}{q-q^{-1}}
\qquad\hbox{and}\qquad
T_iT_{i+1}E_iT_{i+1}^{-1}T_i^{-1} = E_{i+1}.
\end{equation}	
Multiply the second relation in \eqref{Skein} on the left and the right by $E_i$,
and then use the relations in \eqref{rel:Untwisting} to get
$$E_iE_{i+1}E_i
= E_iT_iT_{i+1}E_iT_{i+1}^{-1}T_i^{-1}E_i
=E_iT_{i+1}E_iT_{i+1}^{-1}E_i
=zE_iT_{i+1}^{-1}E_i = E_i,
$$
so that
\begin{equation}\label{SmashE}
E_iE_{i \pm 1} E_i = E_i,
\qquad\hbox{and}\qquad
E_i^2 = \left(1 + \frac{z-z^{-1}}{q-q^{-1}}\right)E_i
\end{equation}
is obtained by multiplying the first equation in \eqref{Skein} by $E_i$ and using 
\eqref{rel:Untwisting}.  Thus, from the first relation in \eqref{rel:Unwrapping},
\begin{equation}
Z_0^{(0)} =  1 + \frac{z-z^{-1}}{q-q^{-1}}
\qquad\hbox{and}\qquad
(T_i - z^{-1})(T_i + q^{-1})(T_i - q) = 0,
\end{equation}
since $(T_i - z^{-1})(T_i + q^{-1})(T_i - q)T_i^{-1}
=(T_i-z^{-1})(T_i^2-(q-q^{-1})T_i-1)T_i^{-1}
=(T_i-z^{-1})(T_i-T_i^{-1}-(q-q^{-1})) = (T_i-z^{-1})(q-q^{-1})(-E_i)
=-(z^{-1}-z^{-1})(q-q^{-1}) = 0$.
The relations
\begin{equation}
E_{i+1} E_i = E_{i+1}T_iT_{i+1}, \qquad\quad
E_i E_{i+1} = T^{-1}_{i+1}T^{-1}_i E_{i+1},
\end{equation}
\begin{equation}
T_i E_{i+1}E_{i} = T_{i+1}^{-1} E_i, \quad \text{ and } \quad E_{i+1} E_i T_{i+1} = E_{i + 1} T_i^{-1},
\end{equation}
follow from the computations
\begin{align*}
&E_{i+1}T_iT_{i+1}
=z(E_{i+1}T_i^{-1}E_{i+1})T_iT_{i+1}
=z(z^{-1}E_{i+1}T_{i+1}^{-1})T_i^{-1}E_{i+1}T_iT_{i+1}
=E_{i+1}E_i, \\
&T_{i+1}^{-1}T_i^{-1}E_{i+1}
=T_{i+1}^{-1}T_i^{-1}(z^{-1}E_{i+1}T_iE_{i+1})
=T_{i+1}^{-1}T_i^{-1}z^{-1}E_{i+1}T_izT_{i+1}E_{i+1}
=E_iE_{i+1}, \\
&T_iE_{i+1}E_i
=T_iE_{i+1}(T_i^{-1}E_iz^{-1})
=z^{-1}T_{i+1}^{-1}E_iT_{i+1}E_iz^{-1}
=T_{i+1}^{-1}zE_iz^{-1}
=T_{i+1}^{-1}E_i, 
\quad\hbox{and}
\\
&E_{i+1}E_i T_{i+1}
=E_{i+1}T_{i+1}^{-1}zE_iT_{i+1}
=zE_{i+1}T_iE_{i+1}T_i^{-1}
=zz^{-1}E_{i+1}T_i^{-1}
= E_{i+1}T_i^{-1}.
\end{align*}

\begin{remark}  A consequence (see \eqref{Genfnadm}) of the defining relations of $W_k$ is the equation
\begin{equation*}
\left(
Z_0^- - \frac{z}{q-q^{-1}} - \frac{u^2}{u^2-1}
\right)
\left(
Z_0^+ + \frac{z^{-1}}{q-q^{-1}} - \frac{u^2}{u^2-1}
\right)E_1
= \frac{-(u^2-q^2)(u^2-q^{-2})}{(u^2-1)(q-q^{-1})^2}E_1,
\end{equation*}
where $Z_0^+$ and $Z_0^-$ are the generating functions
$$Z^+_0 = \sum_{\ell\in \ZZ_{\ge 0}} Z_0^{(\ell)}u^{-\ell}
\qquad\hbox{and}\qquad
Z^-_0 = \sum_{\ell\in \ZZ_{\le 0}} Z_0^{(\ell)}u^{-\ell}.
$$
This means that, unless the parameters $Z_0^{(\ell)}$ are chosen carefully, it is likely
that $E_1=0$ in $W_k$.
\end{remark}

\begin{remark}  From the point of view of Schur-Weyl duality for the affine BMW algebra
(see \cite{OR} and \cite{Ra}) the natural choice of base ring is the center of the 
quantum group corresponding to the orthogonal or symplectic Lie algebra which, by
the (quantum version) of the Harish-Chandra isomorphism, is isomorphic to the 
subring of symmetric Laurent polynomials given by
\begin{equation*}\label{aBMW base ring}
C = \{ z\in \CC[L_1^{\pm1},\ldots, L_r^{\pm1}]^{S_r} \ |\
z(L_1,L_2,\ldots, L_r) = z(L_1^{-1}, L_2,\ldots, L_r)\},
\end{equation*}
where the symmetric group $S_r$ acts by permuting the variables $L_1,\ldots, L_r$. Here
the constants $Z_0^{(\ell)}\in C$ are given, explicitly, by setting the generating functions
$Z_0^+$ and $Z_0^-$ equal, up to a normalization, to
$$
\prod_{i=1}^r \frac{(u-qL_i)} {(u-q^{-1}L_i)}
\cdot \frac{(u-qL_i^{-1} )} {(u-q^{-1}L_i^{-1})}
\qquad\hbox{and}\qquad
\prod_{i=1}^r \frac{(u-q^{-1}L_i)}{(u-qL_i)} 
\cdot \frac{(u-q^{-1}L_i^{-1})}{(u-qL_i^{-1} )},
$$
respectively.  This choice of $C$ and the $Z_0^{(\ell)}$ are the \emph{universal
admissible parameters} for $W_k$.  This point of view will be explained in \cite{DRV}.
\end{remark}

\begin{remark} Careful manipulation of the defining relations of $W_k$ provides
an inductive presentation of $W_k$ as 
$$W_k = W_{k-1}E_{k-1}W_{k-1} 
+W_{k-1}T_{k-1}W_{k-1}
+W_{k-1}T_{k-1}^{-1}W_{k-1}
+\sum_{\ell\in \ZZ} W_{k-1}Y_k^\ell W_{k-1}
$$
(see \cite[Prop.\ 3.16]{GH1} or \cite{Ha3}), and provides that
$$
E_kW_kE_k = W_{k-1}E_k
\qquad\hbox{and}\qquad
\begin{matrix}
W_k &\longrightarrow &W_{k-1}E_k \\
b &\longmapsto &E_kbE_k
\end{matrix}
$$
is a $(W_{k-1}, W_{k-1})$-bimodule homomorphism (see \cite[Prop. 3.17]{GH1}).
These structural facts are important to the understanding of $W_k$ by 
``Jones basic constructions''.
Under the conditions of Theorem \ref{AffineBMWbasis}(a) it is true,
but not immediate from the defining relations,
that the natural homomorphism $W_{k-1}\to W_k$ is injective so that 
$W_{k-1}$ is a subalgebra of $W_k$ (see \cite[Cor.\ 6.15]{GH1}).
\end{remark}

\subsubsection{Quotients of $W_k$}

The \emph{affine Hecke algebra} $H_k$ is the affine BMW algebra $W_k$ 
with the additional relations
\begin{equation}\label{rel:hecke}
	E_i=0, 
	\qquad\hbox{for $i=1,\ldots, k-1$.}
\end{equation}
Fix $b_1,\ldots, b_r\in C$.  The \emph{cyclotomic BMW algebra} $W_{r,k}(b_1,\ldots, b_r)$ 
is the affine BMW algebra $W_k$ with the additional relation 
\begin{equation}\label{cycrelationA}
	(Y_1-b_1)\cdots (Y_1-b_r) = 0.
\end{equation}
The \emph{cyclotomic Hecke algebra} $H_{r,k}(b_1,\ldots, b_r)$ is the affine Hecke algebra
$H_k$ with 
the additional relation \eqref{cycrelationA}.

\begin{remark}\label{Polysubalg}
Since the composite map
$C[Y_1^{\pm1}, \ldots, Y_k^{\pm1}] \to CB_k \to W_k\to H_k$ is injective and the 
last two maps are surjections, it follows that the Laurent polynomial ring 
$C[Y_1^{\pm1}, \ldots, Y_k^{\pm1}] $ is a subalgebra of $CB_k$ and $W_k$.
\end{remark}

\begin{remark}  A consequence of the relation
\eqref{cycrelationA} in $W_{r,k}(b_1,\ldots, b_r)$ is 
\begin{equation}
\left(Z_0^+ + \frac{z^{-1}}{q-q^{-1}}-\frac{u^2}{u^2-1}\right)E_1
=
\left( \frac{z}{q-q^{-1}}+\frac{uz}{u^2-1}\right)
\left(\prod_{j=1}^r \frac{u-b_j^{-1}}{u-b_j}\right)E_1.
\end{equation}
This equation makes the data of the values $b_i$ almost equivalent to the data of
the $Z_0^{(\ell)}$.
\end{remark}

\section{Identities in affine and degenerate affine BMW algebras}
\label{sec:identities}

In \cite{Naz}, Nazarov defined some naturally occurring central elements in the degenerate
affine BMW algebra $\cW_k$ and proved a remarkable recursion for them.  This recursion 
was generalized to analogous central elements in the affine BMW algebra $W_k$ by 
Beliakova-Blanchet \cite{BB}.   In both cases, the recursion was accomplished with an 
involved computation.  In this section, we provide a new proof of the Nazarov and 
Beliakov-Blanchet recursions by lifting them out of the center, to intertwiner-like identities 
in $\cW_k$ and $W_k$ (Propositions \ref{intw}and \ref{Intw}).  These intertwiner-like identities 
for the degenerate affine and affine BMW algebras are reminiscent of the intertwiner identities 
for the degenerate affine and affine Hecke algebras found, for example, in 
\cite[Prop.\ 2.5(c)]{KR} and \cite[Prop.\ 2.14(c)]{Ra1}, respectively.
The central element recursions of \cite{Naz} and \cite{BB} are then obtained by multiplying the intertwiner-like identities by the projectors $e_k$ and $E_k$, respectively.  We have carefully arranged the proofs so that the degenerate affine and the affine cases are exactly in parallel.

\subsection{The degenerate affine case}

Let $u$ be a variable,
\begin{equation}\label{prels1}
u_i^+ = \frac{1}{u-y_i},
\qquad\hbox{and note that}\qquad
u_i^+u_{i+1}^+ 
=\frac{1}{2u-(y_i+y_{i+1})}(u_i^++u_{i+1}^+).
\end{equation}
By \eqref{altedefn} and the definition of $e_i$ in \eqref{rel:e_defn},
$$(u-y_{i+1}) t_{s_i} = t_{s_i}(u-y_i) -(1-e_i)
\quad\text{and}\quad 
(u-y_i)t_{s_i} = t_{s_i}(u-y_{i+1}) + (1-e_i),$$ which 
give
\begin{equation}\label{prels3}
t_{s_i}u_i^+ = u_{i+1}^+t_{s_i}+u_{i+1}^+e_iu_i^+ - u_{i+1}^+u_i^+, 
\quad\hbox{and}\quad
t_{s_i}u_{i+1}^+ = u_i^+t_{s_i}-u_i^+e_iu_{i+1}^+ + u_{i+1}^+u_i^+,
\end{equation}
respectively.

\begin{prop}\label{intw}  In the degenerate affine BMW algebra $\cW_{i+1}$,
\begin{align}
&\left(
e_i\frac{1}{1-y_{i+1}}-t_{s_i}-\frac{1}{2u-(y_i+y_{i+1})}
\right)
\left(
e_i\frac{1}{1-y_{i}}+t_{s_i}-\frac{1}{2u-(y_i+y_{i+1})}
\right)\nonumber\\
&\qquad = 
\frac{-(2u-(y_i+y_{i+1})+1)(2u-(y_i+y_{i+1})-1)}
{(2u-(y_i+y_{i+1}))^2},
\label{genfnlift}
\end{align}
and
\begin{align}
&\left( u_{i+1}^+  +t_{s_i} - e_i \frac{1}{2u-(y_i+y_{i+1})} \right)
-u_i^+\left(u_{i+1}^+ + t_{s_i} - e_i\frac{1}{2u-(y_i+y_{i+1})}\right)u_i^+  \label{recursionlift} \\
&= \left(t_{s_i}u_i^+t_{s_i} + t_{s_i}  - e_i \frac{1}{2u-(y_i+y_{i+1})}\right)
-u_{i+1}^+\left( e_iu_i^+e_i + \epsilon e_i - e_i \frac{1}{2u-(y_i+y_{i+1})}\right) u_{i+1}^+.
 \nonumber
\end{align}
\end{prop}
\begin{proof}
Putting \eqref{prels1} into the first identity in \eqref{prels3} says that if
$$ A = t_{s_i}+\frac{1}{2u-(y_i+y_{i+1})} \quad \text{ and } \quad B=  e_iu_i^++t_{s_i}-\frac{1}{2u-(y_i+y_{i+1})}$$
then 
$$A u_i^+ = u_{i+1}^+ B, \qquad \text{ and }\qquad Ae_i = e_iA$$
follows from \eqref{rel:untwisting} and \eqref{rel:unwrapping}. So
\begin{align*}
&\left(
e_iu_{i+1}^+ - t_{s_i} - \frac{1}{2u-(y_i+y_{i+1})}
\right)
\left(
e_iu_i^+ + t_{s_i} - \frac{1}{2u-(y_i+y_{i+1})}
\right) \\
& \qquad = e_i u_{i+1}^+ B - AB
= e_iAu_i^+ -AB
= Ae_i u_i^+ - AB
= A(e_i u_i^+ - B)\\
&\qquad= 
-\left(t_{s_i} + \frac{1}{2u-(y_i+y_{i+1})}\right)
\left(t_{s_i} - \frac{1}{2u-(y_i+y_{i+1})}\right),
\end{align*}
and multiplying out the right hand side gives \eqref{genfnlift}.

Multiplying the second relation in \eqref{prels3} by $t_{s_i}$ gives
$$u_{i+1}^+ -  t_{s_i}u_{i+1}^+u_i^+ = t_{s_i}u_i^+t_{s_i}-t_{s_i}u_i^+e_iu_i^- 
$$
and again using the relations in \eqref{prels3} gives       
$$u_{i+1}^+ -  u_i^+(t_{s_i} - e_iu_{i+1}^++u_{i+1}^+)u_i^+ 
= t_{s_i}u_i^+t_{s_i}-u_{i+1}^+(t_{s_i} +e_iu_i^+-u_i^+)e_iu_{i+1}^+.
$$
Using \eqref{prels1} and
$$\hbox{adding}\quad
t_{s_i} - e_i\left( \frac{1}{2u-(y_i+y_{i+1})}\right) - \frac{1}{2u-(y_i+y_{i+1})}u_i^+e_iu_{i+1}^+
\quad\hbox{to each side}
$$
gives
\begin{align*}
&\left( u_{i+1}^+  +t_{s_i} - e_i \frac{1}{2u-(y_i+y_{i+1})} \right)
-u_i^+\left(u_{i+1}^+ + t_{s_i} - e_i\frac{1}{2u-(y_i+y_{i+1})}\right)u_i^+ \\
&\qquad= t_{s_i}u_i^+t_{s_i} + t_{s_i}  - e_i \frac{1}{2u-(y_i+y_{i+1})}
-u_{i+1}^+\left( e_iu_i^+ + t_{s_i} - \frac{1}{2u-(y_i+y_{i+1})}\right) e_i u_{i+1}^+ \\
&\qquad= \left(t_{s_i}u_i^+t_{s_i} + t_{s_i}  - e_i \frac{1}{2u-(y_i+y_{i+1})}\right)
-u_{i+1}^+\left( e_iu_i^+e_i + \epsilon e_i - e_i \frac{1}{2u-(y_i+y_{i+1})}\right) u_{i+1}^+,
\end{align*}
completing the proof of \eqref{recursionlift}.
\end{proof}

Introduce notation $z_{i-1}^{(\ell)}e_i$ and the generating function $z_{i-1}(u)e_i$ by
\begin{equation}\label{zidefn}
z_{i-1}(u)e_i = \sum_{\ell \in \ZZ_{\ge 0}} z_{i-1}^{(\ell )} e_iu^{-\ell }
=e_i\left(\sum_{\ell \in \ZZ_{\ge 0}} y_i^{\ell} u^{-\ell }\right)e_i
=e_i\frac{1}{1-y_iu^{-1}} e_i,
\end{equation}
By \cite[Lemma 4.15]{AMR}, or the identity \eqref{grproduct} below, $z_{i-1}^{(\ell)} \in \cW_{i-1}$ 
for $\ell\in \ZZ_{\ge0}$.
If
\begin{equation}\label{prels2}
u_i^- = \frac{1}{u+y_i}
\qquad \hbox{then}\qquad
e_iu_{i+1}^+ = e_iu_i^-,\quad
u_{i+1}^+e_i = u_i^- e_i, \quad
e_iu_i^{\pm}e_i = \frac{z_{i-1}(\pm u)}{u} e_i,
\end{equation}
where, for $i=1$, the last identity is a restatement of the first identity in \eqref{rel:unwrapping}.
The identities \eqref{genfnadm},
\eqref{grrecursion},
and \eqref{grproduct} of
the following theorem are
\cite[Lemma 2.5]{Naz}, \cite[Prop. 4.2]{Naz} and 
\cite[Lemma 3.8]{Naz}, respectively.

\begin{thm} \label{ziprop}  Let $z_{i-1}^{(\ell)}$ and $z_{i-1}(u)$ be as defined in 
\eqref{zidefn}.  
Then $z_{i-1}^{(\ell)}\in Z(\cW_{i-1})$, 
\begin{align}
&\label{genfnadm}
\left( z_{i-1}(-u) - \big( \hbox{$\frac12$} + \epsilon\,u \big) \right)
\left( z_{i-1}(u) - \big( \hbox{$\frac12$} - \epsilon\,u \big) \right) e_i
=
\big( \hbox{$\frac12$} - \epsilon\,u \big)\big( \hbox{$\frac12$} + \epsilon\,u \big) e_i,\\
&\label{grrecursion}
\big(z_i(u) +\epsilon\,u-\hbox{$\frac{1}{2}$}\big)e_{i+1} 
= \left(z_{i-1}(u)+\epsilon\,u-\hbox{$\frac{1}{2}$}\right)
\left( \frac{\big((u+y_i)^2-1\big)(u-y_i)^2}{\big((u-y_i)^2-1\big)(u+y_i)^2}\right)
e_{i+1},
\quad\hbox{and}\\
&\label{grproduct}
(z_{k-1}(u)+\epsilon\,u-\hbox{$\frac12$})e_{i+1} = 
\big(z_0(u)+\epsilon\,u-\hbox{$\frac12$}\big)
\prod_{i=1}^{k-1} \frac{(u+y_i-1)(u+y_i+1)(u-y_i)^2}{(u+y_i)^2(u-y_i+1)(u-y_i-1)}e_{i+1}.
\end{align}
\end{thm}
\begin{proof} 
Since the generators $t_{s_1},\ldots, t_{s_{i-2}}$, $e_1,\ldots, e_{i-2}$ and $y_1,\ldots, y_{i-1}$
of $\cW_{i-1}$
all commute with $e_i$ and $y_i$ it follows that $z_{i-1}^{(\ell)}\in Z(\cW_{i-1})$.

Multiply \eqref{genfnlift} on the right by $e_i$ to get
\eqref{genfnadm}, since $(\half-u)(\half + u) = (\half-\epsilon u)(\half + \epsilon u)$.

Multiplying \eqref{recursionlift} on the left and right by $e_{i+1}$ and using the relations in \eqref{rel:skein}, \eqref{tangle1} and \eqref{tangle2},
$$e_{i+1}t_{s_i}u_i^+t_{s_i}e_{i+1} = e_{i+1}t_{s_i}t_{s_{i+1}}u_i^+t_{s_{i+1}}t_{s_i}e_{i+1}
=e_{i+1}e_iu_i^+e_ie_{i+1}, \text{ and}$$
$$e_{i+1}u_{i+1}^+e_iu_{i+1}^+ e_{i+1}= e_{i+1}u_{i}^-e_iu_{i}^- e_{i+1}= u_{i}^-e_{i+1}e_i e_{i+1}u_{i}^-
=(u_{i}^-)^2e_{i+1},$$
gives
$$\left( \frac{z_i(u)}{u} + \epsilon - \frac{1}{2u} \right) 
	\left( 1- (u_i^+)^2 \right)e_{i+1}
 = \left(\frac{z_{i-1}(u)}{u} + \epsilon - \frac{1}{2u} \right) \left( 1 - (u_i^-)^2 \right)e_{i+1}.$$
So \eqref{grrecursion} follows from
\begin{align*}
\frac{1-(u_i^-)^2}{1-(u_i^+)^2}
&=\frac{1- \big(\frac{1}{u+y_i}\big)^2}{1- \big(\frac{1}{u-y_i}\big)^2} 
=\frac{(u^2+2y_iu+y_i^2-1)(u-y_i)^2}{(u^2-2y_iu+y_i^2-1)(u+y_i)^2}
=\frac{(u+y_i-1)(u+y_i+1)(u-y_i)^2}
{(u-y_i-1)(u-y_i+1)(u+y_i)^2}.
\end{align*}
Finally, 
relation \eqref{grproduct} follows, by induction, from \eqref{grrecursion}.
\end{proof}

\begin{remark}
Using the expansion
$$\frac{1}{u-a} = \frac{u^{-1}}{1-au^{-1}} = \sum_{\ell\in \ZZ_{\ge 1}} a^{\ell-1}u^{-\ell},$$
and taking the coefficient of $u^{-(\ell+1)}$
on each side of the relations in \eqref{prels3} gives
\begin{align}
\label{eq:tyi}
t_{s_i}y_i^\ell &= y_{i+1}^\ell t_{s_i} - (y_{i+1}^{\ell-1}(1-e_i)+y_{i+1}^{\ell-2}(1-e_i)y_i
+\cdots+(1-e_i)y_i^{\ell-1}),  \quad \text{and} \\
\label{eq:tyip1}
t_{s_i}y_{i+1}^\ell &= y_i^\ell t_{s_i} + y_i^{\ell-1}(1-e_i)+y_i^{\ell-2}(1-e_i)y_{i+1}+\cdots
+(1-e_i)y_{i+1}^{\ell-1},
\end{align}
respectively. 
\end{remark}

\begin{remark}
Taking the coefficient of $u^{-s}$ on each side of \eqref{genfnadm}
gives a trivial identity for even $s$ but, for odd $s=2\ell+1$, gives
\begin{equation}\label{admissibility}
\left(2z_{i-1}^{(2\ell+1)}+z_{i-1}^{(2\ell)}-
 (z_{i-1}^{(2\ell)}z_{i-1}^{(0)}-z_{i-1}^{(2\ell-1)}z_{i-1}^{(1)}+\cdots
+z_{i-1}^{(0)}z_{i-1}^{(2\ell)})\right)e_i
=0
\end{equation}
which is the admissibility relation in \cite[Remark 2.11]{AMR} (see also
\cite[(4.6)]{Naz}.)
\end{remark}

\subsection{The affine case}

Let $u$ be a variable,
\begin{equation}\label{Prels1}
U_i^+ = \frac{Y_i}{u-Y_i},
\qquad\hbox{and note that}\qquad
U_i^+U_{i+1}^+ = \frac{Y_iY_{i+1}}{u^2-Y_iY_{i+1}}(U_i^++U_{i+1}^++1).
\end{equation}
By the definition of $E_i$ in \eqref{rel:E_Defn}, 
$$(u-Y_{i+1}) T_i = T_i(u-Y_i) -(q-q^{-1})Y_{i+1}(1-E_i),$$
and, by \eqref{altEdefn},
 $$(u-Y_i)T_i = T_i(u-Y_{i+1}) + (q-q^{-1})(1-E_i)Y_{i+1},$$ 
 so that
\begin{align}\label{eq:TYgenfn}
T_i \frac{1}{u-Y_i} 
&=
\frac{1}{u-Y_{i+1}}T_i 
-(q-q^{-1}) \frac{Y_{i+1}}{u-Y_{i+1}} (1-E_i) \frac{1}{u-Y_i}, \quad\hbox{and}\\
\label{eq:TYp1genfn}
T_i \frac{1}{u-Y_{i+1}}
&= \frac{1}{u-Y_i}T_i + (q-q^{-1})  \frac{1}{u-Y_i} (1-E_i) \frac{Y_{i+1}}{u-Y_{i+1}}.
\end{align}
The relations
\begin{equation}\label{TP1}
\begin{split}T_iU_i^+
&=U_{i+1}^+T_i^{-1}
-(q-q^{-1})U_{i+1}^+(1-E_i)U_{i}^+\\
&=U_{i+1}^+ \left(T_i^{-1} - (q-q^{-1})(1-E_i)U_{i}^+\right),\quad\hbox{and}
\end{split}
\end{equation}
\begin{equation}\label{TP2}
\begin{split}
T_i^{-1}U_{i+1}^+
&=U_i^+T_i-(q-q^{-1})U_i^+E_iU_{i+1}^+ + (q-q^{-1})U_{i+1}^+U_i^+\\
&=U_i^+\left(T_i+(q-q^{-1})(1-E_i)U_{i+1}^+\right)
\end{split}
\end{equation}
are obtained by multiplying \eqref{eq:TYgenfn} and \eqref{eq:TYp1genfn} on the right (resp.\ left)
by $Y_i$ and using the relation $T_iY_i = Y_{i+1}T_i^{-1}$.

\begin{prop}\label{Intw} Let $Q = q-q^{-1}$.  Then, in the affine BMW algebra $W_{i+1}$,
\begin{align}
&\left(
E_i \frac{Y_{i+1}}{u-Y_{i+1}} - \frac{T_i}{Q} - \frac{Y_iY_{i+1}}{u^2-Y_iY_{i+1}}
\right)
\left(
E_i \frac{Y_i}{u-Y_i} + \frac{T_i^{-1}}{Q} - \frac{Y_iY_{i+1}}{u^2-Y_iY_{i+1}}
\right)\qquad\qquad\qquad \nonumber\\
&\qquad\qquad\qquad\qquad
= 
\frac{-(u^2-q^2Y_iY_{i+1})(u^2-q^{-2}Y_iY_{i+1})}{Q^2(u^2-Y_iY_{i+1})^2},
\qquad\hbox{and}
\label{Genfnlift}
\end{align}
\begin{align}
&\left(U_{i+1}^+  + \frac{T_i}{Q} - E_i\frac{Y_iY_{i+1}}{u^2-Y_iY_{i+1}} \right)
-Q^2 (U_i^+ + 1) \left(U_{i+1}^+ + \frac{T_i}{Q} -E_i\frac{Y_iY_{i+1}}{u^2-Y_iY_{i+1}} \right)  U_i^+
\label{Recursionlift} \\
&=
\left(T_i U_i^+T_i^{-1}
+ \frac{T_i}{Q} - E_i\frac{Y_iY_{i+1}}{u^2-Y_iY_{i+1}} \right)
-Q^2 U_{i+1}^+\left(E_{i}U_i^+E_i  +z\frac{E_i}{Q} 
- E_i\frac{Y_iY_{i+1}}{u^2-Y_iY_{i+1}}\right) (U_{i+1}^+ + 1).
\nonumber
\end{align}
\end{prop}
\begin{proof}  Putting \eqref{Prels1} into \eqref{TP1} says that if
$$A = \frac{T_i}{Q} + \frac{Y_iY_{i+1}}{u^2-Y_iY_{i+1}}
\quad\hbox{and}\quad
B= E_iU_i^+ + \frac{T_i^{-1}}{Q} - \frac{Y_iY_{i+1}}{u^2-Y_iY_{i+1}}
$$
then
$$AU_i^+ = U_{i+1}^+B - \frac{Y_iY_{i+1}}{u^2-Y_iY_{i+1}}.
\qquad\hbox{Next,}\quad
AE_i = E_iA
$$
follows from \eqref{rel:Untwisting} and \eqref{rel:Unwrapping}.
So
\begin{align*}
&\left(
E_i \frac{Y_{i+1}}{u-Y_{i+1}} - \frac{T_i}{Q} - \frac{Y_iY_{i+1}}{u^2-Y_iY_{i+1}}
\right)
\left(
E_i \frac{Y_i}{u-Y_i} + \frac{T_i^{-1}}{Q} - \frac{Y_iY_{i+1}}{u^2-Y_iY_{i+1}}
\right)
\\
&= E_i(U_{i+1}^+ B) - AB
= E_i\left(AU_i^+ + \frac{Y_iY_{i+1}}{u^2-Y_iY_{i+1}}\right) - AB
=A(E_iU_i^+-B) + E_i\frac{Y_iY_{i+1}}{u^2-Y_iY_{i+1}} \\
&= - \left(\frac{T_i}{Q} + \frac{Y_iY_{i+1}}{u^2-Y_iY_{i+1}}\right)
\left(\frac{T_i^{-1}}{Q}-\frac{Y_iY_{i+1}}{u^2-Y_iY_{i+1}}\right) 
+ E_i\frac{Y_iY_{i+1}}{u^2-Y_iY_{i+1}},
\end{align*}
and, by \eqref{Skein}, multiplying out the right hand side gives \eqref{Genfnlift}.

Rewrite $T_i^{-1}U_{i+1}^+=U_i^+T_i^{-1}+QU_i^+(1-E_i)(U_{i+1}^++1)$ as
$$
T_i^{-1}U_{i+1}^+-Q(U_{i+1}^++1)U_i^+
=U_i^+T_i^{-1}-QU_i^+E_i(U_{i+1}^++1),
$$
and multiply on the left by $T_i$ to get
\begin{equation}\label{BBstep1}
U_{i+1}^+-QT_i(U_{i+1}^++1)U_i^+
=T_iU_i^+T_i^{-1}-QT_iU_i^+E_i(U_{i+1}^++1).
\end{equation}
Then, since  $T_i = T_i^{-1} + Q(1-E_i)$, equations \eqref{TP2} and \eqref{TP1} imply
\begin{align*}
T_i(U_{i+1}^++1) = Q(U_i^++1)\left(\frac{T_i}{Q}+(1-E_i)U_{i+1}^+\right)
\quad\hbox{and} \quad
T_iU_i^+ = QU_{i+1}^+\left(\frac{T_i^{-1}}{Q}-(1-E_i)U_i^+\right),
\end{align*}
and so \eqref{BBstep1} is
\begin{align}
U_{i+1}^+
&-Q^2 (U_i^+ + 1) \left(\frac{T_i}{Q} + (1-E_i)U_{i+1}^+\right)  U_i^+ 
\label{BBstep2}\\
&=
T_i U_i^+T_i^{-1}-Q^2
U_{i+1}^+\left(\frac{T_i^{-1}}{Q} - (1-E_{i})U_i^+ \right)E_i (U_{i+1}^+ + 1).
\nonumber
\end{align}
Using \eqref{Prels1} and 
$$\hbox{adding}\quad
\frac{T_i}{Q}-E_i\frac{Y_iY_{i+1}}{u^2-Y_iY_{i+1}} 
- Q^2\frac{Y_iY_{i+1}}{u^2-Y_iY_{i+1}}(U_i^++1)E_i(U_{i+1}^++1)
\quad\hbox{to each side}
$$
of \eqref{BBstep2} gives
\begin{align*}
&U_{i+1}^+  + \frac{T_i}{Q} - E_i\frac{Y_iY_{i+1}}{u^2-Y_iY_{i+1}} 
-Q^2 (U_i^+ + 1) \left(U_{i+1}^+ + \frac{T_i}{Q} -E_i\frac{Y_iY_{i+1}}{u^2-Y_iY_{i+1}} \right)  U_i^+ \\
&\ =
T_i U_i^+T_i^{-1}
+ \frac{T_i}{Q} - E_i\frac{Y_iY_{i+1}}{u^2-Y_iY_{i+1}} 
-Q^2 U_{i+1}^+\left(E_{i}U_i^+ + \frac{T_i^{-1}}{Q} 
- \frac{Y_iY_{i+1}}{u^2-Y_iY_{i+1}}\right)E_i (U_{i+1}^+ + 1) \\
&\ =
T_i U_i^+T_i^{-1}
+ \frac{T_i}{Q} - E_i\frac{Y_iY_{i+1}}{u^2-Y_iY_{i+1}} 
-Q^2 U_{i+1}^+\left(E_{i}U_i^+E_i  +z\frac{E_i}{Q} 
- E_i\frac{Y_iY_{i+1}}{u^2-Y_iY_{i+1}}\right) (U_{i+1}^+ + 1),
\end{align*}
completing the proof of \eqref{Recursionlift}.
\end{proof}

Introduce notation $Z_{i-1}^{(\ell)}E_i$ and generating functions $Z_{i-1}^+E_i$
and $Z_{i-1}^-E_i$ by
\begin{align}
Z^+_{i-1}E_i 
&= \sum_{\ell \in \ZZ_{\ge 0}} Z_{i-1}^{(\ell )} E_iu^{-\ell } 
=E_i\left(\sum_{\ell \in \ZZ_{\ge 0}} Y_i^\ell  u^{-\ell }\right)E_i, 
=E_i\frac{1}{1-Y_i u^{-1}} E_i, 
\label{Z+defn}\\
Z^-_{i-1} E_i 
&= \sum_{\ell \in \ZZ_{\ge 0}} Z_{i-1}^{(-\ell )} E_iu^{-\ell }
=E_i\left(\sum_{\ell \in \ZZ_{\ge 0}} Y_i^{-\ell } u^{-\ell }\right)E_i
=E_i\frac{1}{1-Y_i^{-1}u^{-1}} E_i.
\label{Z-defn}
\end{align}
By \cite[Lemma 3.15(1)]{GH1}, or the identity \eqref{Zproduct} below,
$Z_{i-1}^{(\ell)}\in W_{i-1}$ for $\ell\in \ZZ$.
If
\begin{equation}
U_i^- = \frac{Y_i^{-1}}{u-Y_i^{-1}}
\qquad\hbox{then}\qquad
Z_{i-1}^{(0)} = 1 + \frac{z-z^{-1}}{q-q^{-1}},
\end{equation}
by the second relation in \eqref{SmashE}, and
\begin{equation}\label{EP}
E_i U_{i+1}^+ = E_i U_{i}^-,\qquad 
U_{i+1}^+ E_i= U_{i}^- E_i,\qquad
E_i U_i^{\pm} E_i =  (Z^{\pm}_{i-1} - Z_{i-1}^{(0)})E_i,
\end{equation}
where, for $i=1$, the last identity is a restatement of the first identity in \eqref{rel:Unwrapping}.
In the following theorem, the identity \eqref{Genfnadm} is equivalent to \cite[Lemma 2.8, parts (2)and (3)]{GH1} or \cite[Lemma 2.6(4)]{GH2} (see Remark \ref{admtoGH})
and the identity \eqref{Zrecursion}
is found in \cite[Lemma 7.4]{BB}.

\begin{thm} \label{Ziprop}  Let $Z_{i-1}^{(\ell)}$ and the generating functions
$Z_{i-1}^+$ and $Z_{i-1}^-$ be as defined in \eqref{Z+defn} and \eqref{Z-defn}.
Then $Z_{i-1}^{(\ell)}\in Z(W_{i-1})$,
\begin{align}
&\qquad\qquad\left(
Z_{i-1}^- - \frac{z}{q-q^{-1}} - \frac{u^2}{u^2-1}
\right)
\left(
Z_{i-1}^+ + \frac{z^{-1}}{q-q^{-1}} - \frac{u^2}{u^2-1}
\right)E_i
\nonumber\\
&\qquad\qquad\qquad\qquad\qquad\qquad\qquad\qquad\qquad
= \frac{-(u^2-q^2)(u^2-q^{-2})}{(u^2-1)^2(q-q^{-1})^2}E_i,
\label{Genfnadm}\\
&\left(
Z_i^+ +\frac{z^{-1}}{q-q^{-1}}-\frac{u^2}{u^2-1}\right)
E_{i+1} \nonumber \\
&\qquad=
\left(
Z_{i-1}^+ + \frac{z^{-1}}{q-q^{-1}}-\frac{u^2}{u^2-1}\right)
\frac
{(u-Y_i)^2(u-q^{-2}Y_i^{-1})(u-q^2Y_i^{-1})}
{(u-Y_i^{-1})^2(u-q^2Y_i)(u-q^{-2}Y_i)} E_{i+1},\ \hbox{and}
\label{Zrecursion}\\
&\left(Z_{k-1}^+ +\frac{z^{-1}}{q-q^{-1}}-\frac{u^2}{u^2-1}\right)E_{i+1} \nonumber \\
&\qquad=
\left(
Z_0^++\frac{z^{-1}}{q-q^{-1}}-\frac{u^2}{u^2-1}\right)
\left(
\prod_{i=1}^{k-1}
\frac
{(u-Y_i)^2(u-q^{-2}Y_i^{-1})(u-q^2Y_i^{-1})}
{(u-Y_i^{-1})^2(u-q^2Y_i)(u-q^{-2}Y_i)}\right)E_{i+1}.
\label{Zproduct} 
\end{align}
\end{thm}
\begin{proof}
Since the generators $T_1,\ldots, T_{i-2}$, $E_1,\ldots, E_{i-2}$ and $Y_1,\ldots, Y_{i-1}$
of $W_{i-1}$ all commute with $E_i$ and $Y_i$, it follows that $Z_{i-1}^{(\ell)}\in Z(W_{i-1})$.

Multiply \eqref{Genfnlift} on the right by $E_i$ and use
$Z_{i-1}^{(0)} = 1+ (z-z^{-1})/(q-q^{-1})$ to get \eqref{Genfnadm}.

Multiply \eqref{Recursionlift} on the left and right by $E_{i+1}$ and use the relations in \eqref{rel:Untwisting}, \eqref{rel:Unwrapping}, \eqref{SmashE}, and 
$$E_{i+1}T_iU_i^+ T_i^{-1}E_{i+1}
=E_{i+1}T_iT_{i+1}U_i^+ T_{i+1}^{-1}T_i^{-1}E_{i+1}
=E_{i+1}E_iU_i^+ E_iE_{i+1},
$$
to obtain
\begin{align*}
\bigg(Z_i^+&-Z_i^{(0)}+\frac{z}{q-q^{-1}} - \frac{1}{u^2-1}\bigg)\left(1- (q-q^{-1})^2U_i^+(U_i^++1) \right)E_{i+1}\\ 
&=\left(Z_{i-1}^+-Z_{i-1}^{(0)}+\frac{z}{q-q^{-1}} - \frac{1}{u^2-1}\right)\left(1- (q-q^{-1})^2U_i^-(U_i^-+1)\right)E_{i+1} .
\end{align*}
Then \eqref{Zrecursion} follows from
\begin{align*}
&\frac{1-(q-q^{-1})^2U_i^-(U_i^-+1)}{1-(q-q^{-1})^2U_i^+(U_i^++1)}
= \frac{1-(q-q^{-1})^2\frac{Y_i^{-1}}{u-Y_i^{-1}}\left(\frac{Y_i^{-1}}{u-Y_i^{-1}}+1\right)}
{1-(q-q^{-1})^2\frac{Y_i}{u-Y_i}\left(\frac{Y_i}{u-Y_i}+1\right)} \\
&= \frac{((u-Y_i^{-1})^2-(q-q^{-1})^2Y_i^{-1}u)\frac{1}{(u-Y_i^{-1})^2}}
{((u-Y_i)^2-(q-q^{-1})^2Y_iu)\frac{1}{(u-Y_i)^2}} 
= \frac{(u-q^{-2}Y_i^{-1})(u-q^2Y_i^{-1})(u-Y_i)^2}
{(u-q^{-2}Y_i)(u-q^2Y_i)(u-Y_i^{-1})^2} 
\end{align*}
and $Z_i^{(0)}=Z_{i-1}^{(0)} = 1+ (z-z^{-1})/(q-q^{-1})$.  Finally,
relation \eqref{Zproduct} follows, by induction, from \eqref{Zrecursion}.
\end{proof}

\begin{remark}
Taking the coefficient of $u^{-(\ell+1)}$
on each side of \eqref{eq:TYgenfn} and \eqref{eq:TYp1genfn} gives
\begin{align}
\label{eq:TYi}
T_iY_i^\ell &= Y_{i+1}^\ell T_i - (q-q^{-1})(Y_{i+1}^{\ell}(1-E_i)+Y_{i+1}^{\ell-1}(1-E_i)Y_i
+\cdots+Y_{i+1}(1-E_i)Y_i^{\ell-1}), \\
\label{eq:TYip1}
T_iY_{i+1}^\ell &= Y_i^\ell T_i + (q-q^{-1})(Y_i^{\ell-1}(1-E_i)Y_{i+1}+Y_i^{\ell-2}(1-E_i)Y_{i+1}^2
+\cdots +(1-E_i)Y_{i+1}^{\ell}),
\end{align}
respectively, for $\ell\in \ZZ_{\ge 0}$.  Therefore,
\begin{align}
T_iY_i^{-\ell} 
	 &=Y_{i+1}^{-\ell}T_i 
	+  (q-q^{-1})\left(Y_{i+1}^{-(\ell-1)} (1 - E_i) Y_i^{-1} 
	+ \dots +  (1 - E_i) Y_i^{-\ell} \right),\label{eq:TYineg}\\
T_i Y_{i+1}^{-\ell} 
	&= Y_{i}^{-\ell} T_i 
	-(q-q^{-1})\left(Y_i^{-\ell}(1 - E_i)
	+\dots+ Y_i^{-1}(1 - E_i)Y_{i+1}^{-(\ell-1)}\right). \label{eq:TYip1neg}
\end{align}
\end{remark}

\begin{remark}\label{admtoGH} 
Combining \eqref{Genfnadm} and \eqref{Zproduct} yields a formula for 
$Z_{k-1}^-$ in terms of $Z_0^+$ and $Y_1, Y_2, \ldots, Y_{k-1}$.
%
%
%
%
%
Using $Z_{i-1}^{(0)} = 1+\frac{z-z^{-1}}{q-q^{-1}}$, rewrite  \eqref{Genfnadm} as
\begin{align}
\Big(z&Z^{-}_{i-1} -  z^{-1}Z^{+}_{i-1} - (z-z^{-1})Z_{i-1}^{(0)} \Big)E_i \nonumber \\
&= (q-q^{-1}) \left(\frac{1}{u^2 - 1}(Z_{i-1}^++Z_{i-1}^- -Z_{i-1}^{(0)})
- \left(Z^{-}_{i-1} - Z_{i-1}^{(0)}\right) \left(Z^{+}_{i-1} - Z_{i-1}^{(0)}\right)\right)E_i,
\end{align}
and take the coefficient of $u^{-\ell}$ in \eqref{Genfnadm} to get
\begin{equation}\label{Admissibility}
\begin{array}{l} 
\left(zZ_{i-1}^{(-{\ell})} - z^{-1}Z_{i-1}^{({\ell})}\right)E_i \\
\qquad\qquad= (q-q^{-1})
\left(
\begin{array}{l}
Z_{i-1}^{({\ell}-2)}+Z_{i-1}^{({\ell}-4)}+\cdots+Z_{i-1}^{(-({\ell}-2))} \\
- \left(Z_{i-1}^{({\ell}-1)}Z_{i-1}^{(-1)}+Z_{i-1}^{({\ell}-2)}Z_{i-1}^{(-2)}
+\cdots +Z_{i-1}^{(1)}Z_{i-1}^{(-({\ell}-1))}\right)
\end{array}\right)E_i,
\end{array}
\end{equation}
from \cite[Lemma 2.6(4)]{GH2}.
\end{remark}

\section{The center of the affine and degenerate affine BMW algebras}
\label{sec:center}

In this section, we identify the center of $\cW_k$ and $W_k$. Both centers arise as algebras 
of symmetric functions with a ``cancellation property'' (in the language of \cite{Pr}) or ``wheel condition'' (in the language of \cite{FJ+}). In the degenerate case, $Z(\cW_k)$ is the ring of symmetric functions in $y_1, \dots, y_k$ with the $Q$-cancellation property of Pragacz. By \cite[Theorem 2.11(Q)]{Pr},  this is the same ring as the ring generated by the odd power sums, which is the way that Nazarov \cite{Naz} identified $Z(\cW_k)$. 

The cancellation property in the case of $W_k$ is analogous, exhibiting the center of the affine BMW algebra $Z(W_k)$ as a subalgebra of  the ring of symmetric Laurent polynomials.
At the end of this section, in an attempt to make the theory for the affine BMW algebra 
completely analogous to that for the degenerate affine BMW algebra, we have formulated an alternate description of $Z(W_k)$ as a ring generated by ``negative'' power sums.

\subsection{A basis of $\cW_k$}

A \emph{(Brauer) diagram} on $k$ dots is a graph with $k$ dots in the top row, $k$ dots in the bottom
row and $k$ edges pairing the dots.  For example,
\begin{equation}\label{diagex}
d = \beginpicture
\setcoordinatesystem units <0.5cm,0.2cm> 
\setplotarea x from 0.5 to 7, y from 0 to 3    
\linethickness=0.5pt                        
\put{$\bullet$} at 1 -1 \put{$\bullet$} at 1 2
\put{$\bullet$} at 2 -1 \put{$\bullet$} at 2 2
\put{$\bullet$} at 3 -1 \put{$\bullet$} at 3 2
\put{$\bullet$} at 4 -1 \put{$\bullet$} at 4 2
\put{$\bullet$} at 5 -1 \put{$\bullet$} at 5 2
\put{$\bullet$} at 6 -1 \put{$\bullet$} at 6 2
\put{$\bullet$} at 7 -1 \put{$\bullet$} at 7 2
\plot 7 2 4 -1 /
\plot 6 2 6 -1 /
\plot 2 2 1 -1 /
\setquadratic
\plot 1 2 2 1.25 3 2 /
\plot 4 2 4.5 1.25 5 2 /
\plot 2 -1 4.5 0.5 7 -1 /
\plot 3 -1 4 -.25 5 -1 /
\endpicture
\qquad\hbox{
is a Brauer diagram on $7$ dots.}
\end{equation}
Number the vertices of the top row, left to right,  with $1,2,\ldots, k$ and the vertices
in the bottom row, left to right, with $1', 2', \ldots, k'$ so that the diagram in 
\eqref{diagex} can be written
$$d = (13)(21')(45)(66')(74')(2'7')(3'5').$$
The \emph{Brauer algebra}  is the vector space 
\begin{equation}\label{Brauerbasis}
\cW_{1,k}\quad\hbox{with basis }\qquad	D_k = \{ \text{ diagrams on $k$ dots } \},
\end{equation}
and product given by stacking diagrams and changing each closed loop to $x$.  For example,
$$\hbox{if}\qquad
{\beginpicture
\setcoordinatesystem units <0.5cm,0.2cm> 
\setplotarea x from 0 to 7, y from 0 to 3    
\linethickness=0.5pt
\put{$d_1 = $} at -.5 0
\put{$\bullet$} at 1 -1.5 \put{$\bullet$} at 1 1.5
\put{$\bullet$} at 2 -1.5 \put{$\bullet$} at 2 1.5
\put{$\bullet$} at 3 -1.5 \put{$\bullet$} at 3 1.5
\put{$\bullet$} at 4 -1.5 \put{$\bullet$} at 4 1.5
\put{$\bullet$} at 5 -1.5 \put{$\bullet$} at 5 1.5
\put{$\bullet$} at 6 -1.5 \put{$\bullet$} at 6 1.5
\put{$\bullet$} at 7 -1.5 \put{$\bullet$} at 7 1.5
\setquadratic
\plot 1 1.5 2.5 0.75 4 1.5 /
\plot 5 1.5 5.5 0.75 6 1.5 /
\plot 2 -1.5 2.5 -.75 3 -1.5 /
\plot 5 -1.5 6 -0.25 7 -1.5 /
\setlinear
\plot 2 1.5 1 -1.5 /
\plot 3 1.5 4 -1.5 /
\plot 7 1.5 6 -1.5 /
\endpicture}
\quad\hbox{and}\quad
{\beginpicture
\setcoordinatesystem units <0.5cm,0.2cm> 
\setplotarea x from 0 to 7, y from 0 to 3    
\linethickness=0.5pt
\put{$d_2 = $} at -.5 0
\put{$\bullet$} at 1 1.5 \put{$\bullet$} at 1 -1.5
\put{$\bullet$} at 2 1.5 \put{$\bullet$} at 2 -1.5
\put{$\bullet$} at 3 1.5 \put{$\bullet$} at 3 -1.5
\put{$\bullet$} at 4 1.5 \put{$\bullet$} at 4 -1.5
\put{$\bullet$} at 5 1.5 \put{$\bullet$} at 5 -1.5
\put{$\bullet$} at 6 1.5 \put{$\bullet$} at 6 -1.5
\put{$\bullet$} at 7 1.5 \put{$\bullet$} at 7 -1.5
\setquadratic
\plot 2 1.5 3 0.25 4 1.5 /
\plot 5 1.5 6 0.25 7 1.5 /
\plot 5 -1.5 5.5 -1 6 -1.5 /
\plot 3 -1.5 5.5 -.2 7 -1.5 /
\setlinear
\plot 1 1.5 1 -1.5 /
\plot 3 1.5 4 -1.5 /
\plot 6 1.5 2 -1.5 /
\endpicture}
\quad\hbox{then}
$$
\medskip
\begin{equation}
{\beginpicture
\setcoordinatesystem units <0.4cm,0.15cm> 
\setplotarea x from 0 to 7, y from 0 to 3    
\linethickness=0.5pt
\put{$d_1d_2 = $} at -.75 0
\put{$\bullet$} at 1 1 \put{$\bullet$} at 1 4
\put{$\bullet$} at 2 1 \put{$\bullet$} at 2 4
\put{$\bullet$} at 3 1 \put{$\bullet$} at 3 4
\put{$\bullet$} at 4 1 \put{$\bullet$} at 4 4
\put{$\bullet$} at 5 1 \put{$\bullet$} at 5 4
\put{$\bullet$} at 6 1 \put{$\bullet$} at 6 4
\put{$\bullet$} at 7 1 \put{$\bullet$} at 7 4
\put{$\bullet$} at 1 -1 \put{$\bullet$} at 1 -4
\put{$\bullet$} at 2 -1 \put{$\bullet$} at 2 -4
\put{$\bullet$} at 3 -1 \put{$\bullet$} at 3 -4
\put{$\bullet$} at 4 -1 \put{$\bullet$} at 4 -4
\put{$\bullet$} at 5 -1 \put{$\bullet$} at 5 -4
\put{$\bullet$} at 6 -1 \put{$\bullet$} at 6 -4
\put{$\bullet$} at 7 -1 \put{$\bullet$} at 7 -4
\setquadratic
\plot 1 4 2.5 3 4 4 /
\plot 5 4 5.5 3.25 6 4 /
\plot 2 1 2.5 1.75 3 1 /
\plot 5 1 6 1.75 7 1 /
\plot 2 -1 3 -1.75 4 -1 /
\plot 5 -1 6 -1.75 7 -1 /
\plot 3 -4 5.3 -2.8 7 -4 /
\plot 5 -4 5.5 -3.5 6 -4 /
\setlinear
\plot 2 4 1 1 /
\plot 3 4 4 1 /
\plot 7 4 6 1 /
\plot 1 -1 1 -4 /
\plot 3 -1 4 -4 /
\plot 6 -1 2 -4 /
\setdashpattern <.03cm,.03cm>
\plot 1 1 1 -1 /
\plot 2 1 2 -1 /
\plot 3 1 3 -1 /
\plot 4 1 4 -1 /
\plot 5 1 5 -1 /
\plot 6 1 6 -1 /
\plot 7 1 7 -1 /
\endpicture}
~=
~~x~ {\beginpicture
\setcoordinatesystem units <0.4cm,0.25cm> 
\setplotarea x from 1 to 7, y from 0 to 3    
\linethickness=0.5pt                        
\put{$\bullet$} at 1 -1 \put{$\bullet$} at 1 2
\put{$\bullet$} at 2 -1 \put{$\bullet$} at 2 2
\put{$\bullet$} at 3 -1 \put{$\bullet$} at 3 2
\put{$\bullet$} at 4 -1 \put{$\bullet$} at 4 2
\put{$\bullet$} at 5 -1 \put{$\bullet$} at 5 2
\put{$\bullet$} at 6 -1 \put{$\bullet$} at 6 2
\put{$\bullet$} at 7 -1 \put{$\bullet$} at 7 2
\plot 2 2 1 -1 /
\plot 3 2 4 -1 /
\plot 7 2 2 -1 /
\setquadratic
\plot 1 2 2.5 1.25 4 2 /
\plot 5 2 5.5 1.5 6 2 /
\plot 3 -1 5.5 0 7 -1 /
\plot 5 -1 5.5 -.5 6 -1 /
\endpicture}.
\end{equation}
The Brauer algebra
is generated by
\begin{equation}
e_{i} =
{\beginpicture
\setcoordinatesystem units <0.5cm,0.2cm> 
\setplotarea x from 2 to 5.5, y from -1 to 4   
\linethickness=0.5pt                        
\put{$\bullet$} at 2 -1 \put{$\bullet$} at 2 2 \plot 2 -1 2 2 /
\put{$\cdots$} at 3 .5
\put{$\bullet$} at 4 -1  \put{$\bullet$} at 4 2 \plot 4 -1 4 2 /
\put{$\bullet$} at 5 -1 \put{$\bullet$} at 5 2
\put{$\bullet$} at 6 -1 \put{$\bullet$} at 6 2
\put{$\scriptstyle{i}$}[b] at 5 3.45
\put{$\scriptstyle{i+1}$}[b] at 6 3.5
\put{$\bullet$} at 7 -1  \put{$\bullet$} at 7 2 \plot 7 -1 7 2 /
\put{$\cdots$} at 8 .5
\put{$\bullet$} at 9 -1  \put{$\bullet$} at 9 2 \plot 9 -1 9 2 /
\setquadratic
\plot 5 2 5.5 1.5 6 2 /
\plot 5 -1 5.5 -.5 6 -1 /
\endpicture} ,
\qquad\qquad
s_i = 
{\beginpicture
\setcoordinatesystem units <0.5cm,0.2cm> 
\setplotarea x from 1 to 5.5, y from -1 to 4   
\linethickness=0.5pt                        
\put{$\bullet$} at 1 -1 \put{$\bullet$} at 1 2 \plot 1 -1 1 2 /
\put{$\cdots$} at 2 .5
\put{$\bullet$} at 3 -1  \put{$\bullet$} at 3 2 \plot 3 -1 3 2 /
\put{$\bullet$} at 4 -1 \put{$\bullet$} at 4 2
\put{$\bullet$} at 5 -1 \put{$\bullet$} at 5 2
\put{$\scriptstyle{i}$}[b] at 4 3.5
\put{$\scriptstyle{i+1}$}[b] at 5 3.45
\put{$\bullet$} at 6 -1 \put{$\bullet$} at 6 2  \plot 6 -1 6 2 /
\put{$\cdots$} at 7 .5
\put{$\bullet$} at 8 -1 \put{$\bullet$} at 8 2  \plot 8 -1 8 2 /
\plot 4 2 5 -1 /
\plot 5 2 4 -1 /
\endpicture},
\qquad\quad
1\le i\le k-1.
\end{equation}
Setting 
$$
x = z_0^{(0)}
\qquad\hbox{and}\qquad 
s_i=\epsilon t_{s_i}
$$ realizes the Brauer algebra 
as a subalgebra of the degenerate affine BMW algebra $\cW_k$. 
The Brauer algebra is also the quotient of $\cW_k$ by $y_1 = 0$ and, hence,
can be viewed as the degenerate cyclotomic BMW algebra $\cW_{1,k}(0)$. 

\begin{thm}
\label{thm:W_basis}
Let $\cW_k$ be the degenerate affine BMW algebra and let 
$\cW_{r,k}(b_1,\ldots, b_r)$ be the degenerate cyclotomic BMW algebra as defined in
\eqref{rel:untwisting}-\eqref{rel:unwrapping} and \eqref{cycrelation}, respectively. 
For $n_1,\ldots, n_k\in \ZZ_{\geq 0}$ and a diagram $d$ on $k$ dots let
$$d^{n_1,\ldots, n_k} 
= y_{i_1}^{n_1}\cdots y_{i_\ell}^{n_\ell} d y_{i_{\ell+1}}^{n_{\ell+1}}\cdots y_{i_k}^{n_k},$$
where, in the lexicographic ordering of the edges $(i_1,j_1), \ldots, (i_k,j_k)$ of $d$, 
$i_1,\ldots, i_\ell$ are in the top row of $d$ and $i_{\ell+1}, \ldots, i_k$
are in the bottom row of $d$.  Let $D_k$ be the set of diagrams on $k$ dots, as in \eqref{Brauerbasis}.
\item[(a)] 
If $\kappa_0,\kappa_1\in C$ and 
\begin{equation}\label{admissibility}
\left( z_{0}(-u) - \big( \hbox{$\frac12$} + \epsilon\,u \big) \right)
\left( z_{0}(u) - \big( \hbox{$\frac12$} - \epsilon\,u \big) \right)
=
\big( \hbox{$\frac12$} - \epsilon\,u \big)\big( \hbox{$\frac12$} + \epsilon\,u \big)
\end{equation}
then $\{d^{n_1,\ldots, n_k}\ |\ d\in D_k,\  n_1,\ldots, n_k\in \ZZ_{\geq 0} \}$
is a $C$-basis of $\cW_k$.
\item[(b)] 
If $\kappa_0,\kappa_1\in C$, \eqref{admissibility} holds, and 
\begin{equation}\label{cycadmissibility}
\left(z_0(u) + u-\hbox{$\frac12$}\right)
= (u-\hbox{$\frac12$}(-1)^r)\left(\prod_{i=1}^r \frac{u+b_i}{u-b_i}\right)
\end{equation}
then $\{d^{n_1,\ldots, n_k}\ |\ d\in D_k,\  0\le n_1,\ldots, n_k\le r-1 \}$
is a $C$-basis of $\cW_{r,k}(b_1, \dots, b_r)$.
 \end{thm}

\smallskip\noindent
 Part (a) of Theorem \ref{thm:W_basis} is \cite[Theorem 4.6]{Naz} (see also \cite[Theorem 2.12]{AMR}) and part (b) is  
\cite[Prop.\ 2.15 and Theorem 5.5]{AMR}. We refer to these references for the proof, 
remarking only that 
one key point in showing that $\{d^{n_1,\ldots, n_k}\ |\ d\in D_k, n_1,\ldots, n_k\in \ZZ_{\geq 0} \}$
spans $\cW_k$ is that if $(i,j)$ is a top-to-bottom edge in $d$, then
\begin{equation}
\label{yd-vert-terms}
	y_i d = d y_j + \text{(terms with fewer crossings)},
\end{equation}
and if $(i,j)$ is a top-to-top edge in $d$ then 
\begin{equation}
\label{yd-horiz-terms}
	y_i d = -y_j d + \text{(terms with fewer crossings)}.
\end{equation}
This is illustrated in the affine case in \eqref{basisexample}.

\subsection{The center of $\cW_k$}
\label{sec:deg_center}

The degenerate affine BMW algebra is the algebra $\cW_k$ over $C$ defined in 
Section \ref{s:degenaff} and the polynomial ring $C[y_1,\ldots, y_k]$ is a subalgebra of $\cW_k$ (see Remark \ref{polysubalg}).  
The symmetric group $S_k$ acts on $C[y_1,\ldots, y_k]$ by permuting the
variables and the ring of symmetric functions is
$$C[y_1,\ldots, y_k]^{S_k}
=\{ f \in C[y_1,\ldots, y_k]\ |\ \hbox{$wf=f$, for $w\in S_k$}\}.
$$
A classical fact (see, for example, \cite[Theorem 3.3.1]{Kl}) is that the
center of the degenerate affine Hecke algebra $\cH_k$ is 
$$Z(\cH_k) = C[y_1,\ldots, y_k]^{S_k}.$$
Theorem \ref{degBMWcenter} gives an analogous characterization of the center of the
degenerate affine BMW algebra.

\begin{thm}\label{degBMWcenter} 
The center of the degenerate affine BMW algebra $\cW_k$ is 
$$\cR_k =  \{ f \in C[y_1,\ldots, y_k]^{S_k}\ |\ 
f(y_1,-y_1, y_3, \ldots, y_k) =  f(0, 0, y_3, \ldots, y_k) \}.$$

\end{thm}

\begin{proof}

\noindent \emph{Step 1: $f\in \cW_k$ commutes with all $y_i$ $\Leftrightarrow$ $f \in C[y_1, \dots, y_k]$:}\\
Assume $f \in \cW_k$ and write
$$f = \sum c_d^{n_1,\ldots, n_k} d^{n_1,\ldots, n_k}$$
in terms of the basis in Theorem \ref{thm:W_basis}. 
Let $d\in D_k$ with the maximal number of crossings
such that $c_d^{n_1,\ldots, n_k}\ne 0$ and, using the notation before \eqref{Brauerbasis},
suppose there is an edge $(i,j)$ of $d$ such that $j\ne i'$.  Then,
by \eqref{yd-vert-terms} and \eqref{yd-horiz-terms},
$$\text{the coefficient of} \quad y_id^{n_1,\ldots, n_k} \quad\hbox{in}\quad
y_i f \quad\hbox{is}\quad c_d^{n_1,\ldots, n_k}$$
and 
$$\text{the coefficient of} \quad y_id^{n_1,\ldots, n_k} \quad\hbox{in}\quad
f y_i \quad\hbox{is}\quad 0.$$
If $y_i f= f y_i$, it follows that there is no such edge, and so $d=1$. Thus $f \in C[y_1,\ldots, y_k]$. 

\noindent
Conversely, if $f \in C[y_1,\ldots, y_k]$ then $y_i f = f y_i$.

~

\noindent \emph{Step 2: $f\in C[y_1,\ldots, y_k]$ commutes with all $t_{s_i}$ $\Leftrightarrow$ $f \in \cR_k$:}\\
Assume $f\in C[y_1, \dots, y_k]$ and write 
\begin{equation*}
f = \sum_{a,b \in \ZZ_{\geq 0} } y_1^a y_2^b f_{a,b}, \qquad \text{ where } f_{a,b} \in C[y_3, \dots, y_k].
\end{equation*}
Then $\displaystyle{
f(0, 0, y_3, \dots, y_k) = \sum_{a,b \in \ZZ_{\geq 0}} f_{a,b}
}$
and
\begin{equation}
\label{yyminus}
f(y_1, -y_1, y_3, \dots, y_k) = \sum_{a,b \in \ZZ_{\geq 0}} (-1)^b y_1^{a+b} f_{a,b}
= \sum_{\ell \in \ZZ_{\geq 0 }} y_1^\ell \left( \sum_{b=0}^\ell (-1)^b f_{\ell-b, b}\right).
\end{equation}
By direct computation using  \eqref{eq:tyi} and \eqref{eq:tyip1},
\begin{align*}
t_{s_1} y_1^a y_2^b 
=s_1(y_1^a y_1^b) t_{s_1} - \frac{y_1^a y_2^b - s_1(y_1^ay_2^b)}{y_1 - y_2}
+(-1)^{a} \sum_{r=1}^{a+b} (-1)^r y_1^{a+b-r} e_1 y_1^{r-1},
\end{align*}
and it follows that 
\begin{equation}\label{ts1pastf}
t_{s_1} f =  (s_1 f) t_{s_1} - \frac{f - s_1 f }{y_1 - y_2} 
+ \sum_{\ell \in \ZZ_{> 0}}\left(
\left( \sum_{r=1}^{\ell} (-1)^r y_1^{\ell-r} e_1 y_1^{r-1} \right) 
 \left(\sum_{b=0}^{\ell} (-1)^{\ell-b}f_{\ell-b,b} \right) \right).
 \end{equation}
 Thus,  if  $f(y_1, -y_1, y_3, \dots, y_k) = f(0, 0, y_3, \dots, y_k)$, then
\begin{equation}\label{e-sym-cancellation}
\sum_{b=0}^\ell (-1)^b f_{\ell-b, b} = 0, \qquad\hbox{for $\ell \neq 0$.}
\end{equation}
Hence, if  $f\in C[y_1, \dots, y_k]^{S_k}$ and $f(y_1, -y_1, y_3, \dots, y_k) = f(0, 0, y_3, \dots, y_k)$ then $s_1f=f$ and, by \eqref{yyminus}, \eqref{e-sym-cancellation} holds so that,
by \eqref{ts1pastf}, $t_{s_1} f = f t_{s_1}.$ Similarly, $f$ commutes with all $t_{s_i}$.

\noindent
Conversely,  if $f\in C[y_1, \dots, y_k]$ and $t_{s_i} f = f t_{s_i}$ then
$$s_i f = f \qquad \text{ and } \qquad \sum_{b=0}^{\ell} (-1)^{\ell-b}f_{\ell-b,b} =0, \qquad
\hbox{for $\ell \neq 0$,}$$
so that $f\in C[y_1, \dots, y_k]^{S_k}$ and $f(y_1, -y_1, y_3, \dots, y_k) = f(0,0, y_3, \dots , y_k)$.

It follows from \eqref{rel:e_defn} that $\cR_k = Z(\cW_k)$.
\end{proof}

The \emph{power sum symmetric functions} $p_i$ are given by 
$$p_i = y_1^i+y_2^i+\cdots+y_k^i, \qquad\hbox{for $i\in \ZZ_{>0}$.}$$
The \emph{Hall-Littlewood polynomials} (see \cite[Ch.\ III (2.1)]{Mac}) are given by 
$$P_\lambda(y;t) = P_\lambda(y_1,\ldots, y_k; t) = \frac{1}{v_\lambda(t)}
\sum_{w\in S_k} w\left(y_1^{\lambda_1}\cdots y_k^{\lambda_k}\prod_{1\le i<j\le k}
\frac{x_i-tx_j}{x_i-x_j}\right),$$
where $v_\lambda(t)$ is a normalizing constant (a polynomial in $t$) so that 
the coefficient of  $y_1^{\lambda_1}\cdots y_k^{\lambda_k}$ in $P_\lambda(y;t)$ is equal to $1$.
The \emph{Schur Q-functions} (see \cite[Ch.\ III (8.7)]{Mac}) are
$$Q_\lambda = \begin{cases}
0, &\hbox{if $\lambda$ is not strict}, \\
2^{\ell(\lambda)}P_\lambda(y;-1), &\hbox{if $\lambda$ is strict,}
\end{cases}$$
where $\ell(\lambda)$ is the number of (nonzero) parts of $\lambda$ and the partition
$\lambda$ is \emph{strict} if all its (nonzero) parts are distinct.
Let $\cR_k$ be as in Theorem \ref{degBMWcenter}.  Then
(see \cite[Cor. 4.10]{Naz}, \cite[Theorem 2.11(Q)]{Pr} and \cite[Ch.\ III \S 8]{Mac})
\begin{equation}\label{oddpower}
\cR_k = C[p_1,p_3, p_5, \ldots] = 
\hbox{$C$span-}\{ Q_\lambda\ |\ \hbox{$\lambda$ is strict}\}.
\end{equation}
More generally, let $r\in \ZZ_{>0}$ and let $\zeta$ be a primitive $r$th root of unity.
Define
$$\cR_{r,k} =  \{ f \in \ZZ[\zeta][y_1,\ldots, y_k]^{S_k}\ |\ 
f(y_1,\zeta y_1,\ldots, \zeta^{r-1}y_1, y_{r+1}, \ldots, y_k) =  f(0, 0, \ldots, 0, y_{r+1}, \ldots, y_k) \}.
$$
Then
\begin{equation}\label{RrkA}
\cR_{r,k}\otimes_{\ZZ[\zeta]} \QQ(\zeta) = \QQ(\zeta)[p_i\ |\ i\ne 0 \bmod r],
\end{equation}
and 
\begin{equation}\label{RrkB}
\hbox{$\cR_{r,k}$ has $\ZZ[\zeta]$-basis}\qquad
\{ P_\lambda(y;\zeta)\ |\ \hbox{$m_i(\lambda)<r$ and $\lambda_1\le k$}\},
\end{equation}
where $m_i(\lambda)$ is the number parts of size $i$ in $\lambda$.  The ring $\cR_{r,k}$
is studied in \cite{Mo},  \cite{LLT}, \cite[Ch.\ III Ex.\ 5.7 and Ex.\ 7.7]{Mac}, \cite{To}, \cite{FJ+}, 
and others.   The proofs of \eqref{RrkA} and \eqref{RrkB} follow from
\cite[Ch.\ III Ex.\ 7.7]{Mac}, \cite[Lemma 2.2 and following remarks]{To} and the arguments in the 
proofs of \cite[Lemma 3.2 and Proposition 3.5]{FJ+}.

\begin{remark}  The left ideal of $\cW_2$ generated by $e_1$ is $C[y_1]e_1$. This is an infinite dimensional (generically irreducible) $\cW_2$-module on which $Z(\cW_2)$ acts by constants. Thus, as noted by \cite[par.\ before Ex.\ 2.17]{AMR}, it follows that $\cW_2$ is not finitely generated 
as a $Z(\cW_2)$-module. 
\end{remark}

\subsection{A basis of $W_k$}

An affine tangle has $k$ strands and a flagpole just as in the case of an affine braid, but there 
is no restriction that a strand must connect an upper vertex to a lower vertex. 
Let $X^{\varepsilon_1}$ and $T_i$  be the affine braids given in (\ref{braidfig1}) and 
let 
\begin{equation}
E_i = 
\beginpicture
\setcoordinatesystem units <.5cm,.5cm>         
\setplotarea x from -5 to 3.5, y from -1 to 1    
\put{$\bullet$} at -3 0.75      %
\put{$\bullet$} at -2 0.75      %
\put{$\bullet$} at -1 0.75      %
\put{$\bullet$} at  0 0.75      
\put{$\bullet$} at  1 0.75      %
\put{$\bullet$} at  2 0.75      %
\put{$\bullet$} at  3 0.75      %
\put{$\bullet$} at -3 -0.75          %
\put{$\bullet$} at -2 -0.75          %
\put{$\bullet$} at -1 -0.75          %
\put{$\bullet$} at  0 -0.75          
\put{$\bullet$} at  1 -0.75          %
\put{$\bullet$} at  2 -0.75          %
\put{$\bullet$} at  3 -0.75          %
\plot -4.5 1.25 -4.5 -1.25 /
\plot -4.25 1.25 -4.25 -1.25 /
\ellipticalarc axes ratio 1:1 360 degrees from -4.5 1.25 center 
at -4.375 1.25
\put{$*$} at -4.375 1.25  
\ellipticalarc axes ratio 1:1 180 degrees from -4.5 -1.25 center 
at -4.375 -1.25 
\plot -3 0.75  -3 -0.75 /
\plot -2 0.75  -2 -0.75 /
\plot -1 0.75  -1 -0.75 /
\plot  2 0.75   2 -0.75 /
\plot  3 0.75   3 -0.75 /
\plot 0.3 0.25  0.7 0.25 /
\plot 0.3 -0.25    0.7 -0.25   /
\setquadratic
\plot  0.7 0.25  0.95 0.45  1 0.75 /
\plot  0 0.75  0.05 0.45  0.3 0.25 /
\plot  0.7 -0.25  0.95 -0.45  1 -0.75 /
\plot  0 -0.75  0.05 -0.45  0.3 -0.25 /
\endpicture.
\end{equation}
Goodman and Hauchschild \cite[Cor. 6.14(b)]{GH1} 
have shown that the affine BMW algebra $W_k$ is the algebra of linear combinations 
of tangles generated 
by $X^{\varepsilon_1}, T_1, \ldots, T_{k-1}, E_1, \ldots, E_{k-1}$ and the relations 
\eqref{rel:Untwisting}, \eqref{rel:Unwrapping}  and 
\eqref{Skein} expressed in the form 
\begin{equation}
\beginpicture
\setcoordinatesystem units <.5cm,.5cm>         
\setplotarea x from -1 to 1, y from -1 to 1    
\setquadratic
\plot  -.5 -.75  -.45 -0.45  -0.1 -0.1 /
\plot  .1 .15  .45 .45  .5 .75 /
\plot -.5 .75  -.45 .45  0 0  .45 -.45  .5 -.75 /
\endpicture
-
\beginpicture
\setcoordinatesystem units <.5cm,.5cm>         
\setplotarea x from -1 to 1, y from -1 to 1    
\setquadratic
\plot  -.5 .75  -.45 0.45  -0.1 0.1 /
\plot  .1 -.15  .45 -.45  .5 -.75 /
\plot -.5 -.75  -.45 -.45  0 0  .45 .45  .5 .75 /
\endpicture
~=~
(q-q^{-1})\left(
\beginpicture
\setcoordinatesystem units <.5cm,.5cm>         
\setplotarea x from -1 to 1, y from -1 to 1    
\plot -0.5 .75   -0.5 -.75 /
\plot  0.5 .75    0.5 -.75 /
\endpicture
-
\beginpicture
\setcoordinatesystem units <.5cm,.5cm>         
\setplotarea x from -1 to 1, y from -1 to 1    
\plot -0.2 0.25  0.2 0.25 /
\plot -0.2 -0.25    0.2 -0.25   /
\setquadratic
\plot  0.2 0.25  0.45 0.45  0.5 0.75 /
\plot  -.5 0.75  -.45 0.45  -0.2 0.25 /
\plot  0.2 -0.25  0.45 -0.45  0.5 -0.75 /
\plot  -.5 -0.75  -.45 -0.45  -0.2 -0.25 /
\endpicture
\right)
\end{equation}
\begin{equation}
\beginpicture
\setcoordinatesystem units <.5cm,.5cm>         
\setplotarea x from -1 to 2, y from -1 to 1    
\plot 1.5 .75  1.5 -0.75 /
\plot 0.8 -1.25  1.2 -1.25 /
\plot 0.8 1.25    1.2 1.25   /
\setquadratic
\plot  1.2 -1.25  1.45 -1.05  1.5 -0.75 /
\plot  0.5 -0.75  .55 -1.05  0.8 -1.25 /
\plot  1.2 1.25  1.45 1.05  1.5 0.75 /
\plot  .5 0.75  .55 1.05  0.8 1.25 /
\setquadratic
\plot  -.5 -.75  -.45 -0.45  -0.1 -0.1 /
\plot  .1 .15  .45 .45  .5 .75 /
\plot -.5 .75  -.45 .45  0 0  .45 -.45  .5 -.75 /
\endpicture
=z\;
\beginpicture
\setcoordinatesystem units <.5cm,.5cm>         
\setplotarea x from -0.5 to 0.5, y from -1 to 1    
\plot  0 .75    0 -.75 /
\endpicture
\qquad\hbox{and}\qquad
\beginpicture
\setcoordinatesystem units <.5cm,.5cm>         
\setplotarea x from -1 to 2, y from -1 to 1    
\plot 1.5 .75  1.5 -0.75 /
\plot 0.8 -1.25  1.2 -1.25 /
\plot 0.8 1.25    1.2 1.25   /
\setquadratic
\plot  1.2 -1.25  1.45 -1.05  1.5 -0.75 /
\plot  0.5 -0.75  .55 -1.05  0.8 -1.25 /
\plot  1.2 1.25  1.45 1.05  1.5 0.75 /
\plot  .5 0.75  .55 1.05  0.8 1.25 /
\setquadratic
\plot  -.5 .75  -.45 0.45  -0.1 0.1 /
\plot  .1 -.15  .45 -.45  .5 -.75 /
\plot -.5 -.75  -.45 -.45  0 0  .45 .45  .5 .75 /
\endpicture
= z^{-1}
\beginpicture
\setcoordinatesystem units <.5cm,.5cm>         
\setplotarea x from -0.5 to 0.5, y from -1 to 1    
\plot  0 .75    0 -.75 /
\endpicture
\end{equation}
\begin{equation}
\beginpicture
\setcoordinatesystem units <.5cm,.5cm>         
\setplotarea x from -4 to 1.5, y from -2 to 2    
\put{$\hbox{$\ell$ loops\ }\Bigg\{$} at -3.8 .4   %
\put{$\bullet$} at  0 2      
\put{$\bullet$} at  0 -1.5          
\plot -1.5 2.2 -1.5 1.12 /
\plot -1.5 .88 -1.5 0.12 /
\plot -1.5 -0.12 -1.5 -0.88 /
\plot -1.5 -1.12 -1.5 -1.75 /
\plot -1.25 2.2 -1.25 1.12 /
\plot -1.25 .88 -1.25 0.12 /
\plot -1.25 -0.12 -1.25 -0.88 /
\plot -1.25 -1.12 -1.25 -1.75 /
\ellipticalarc axes ratio 1:1 360 degrees from -1.5 2.2 center 
at -1.375 2.2
\put{$*$} at -1.375 2.2  
\ellipticalarc axes ratio 1:1 180 degrees from -1.5 -1.75 center 
at -1.375 -1.75 
\plot  1 2   1 -1.5 /
\setlinear
\plot -0.3 1.5  -1.1 1.5 /
\ellipticalarc axes ratio 2:1 180 degrees from -1.65 1.5  center 
at -1.65 1.25 
\plot -1.65 1  -1.1 1 /
\ellipticalarc axes ratio 2:1 -180 degrees from -1.1 1  center 
at -1.1 .75 
\ellipticalarc axes ratio 2:1 180 degrees from -1.65 .5  center 
at -1.65 .25 
\ellipticalarc axes ratio 2:1 -180 degrees from -1.1 0  center 
at -1.1 -0.25 
\plot -1.65 0  -1.1 0 /
\ellipticalarc axes ratio 2:1 180 degrees from -1.65 -.5  center 
at -1.65 -.75 
\plot -1.65 -1  -0.3 -1 /
\setquadratic
\plot  -0.3 1.5  -0.05 1.7  -0 2 /
\plot  -0.3 -1  -0.05 -1.2  -0 -1.5 /
\ellipticalarc axes ratio 1:1 180 degrees from 1 2 center 
at 0.5 2
\ellipticalarc axes ratio 1:1 180 degrees from 0 -1.5 center 
at 0.5 -1.5
\endpicture
= Z_0^{(\ell)}
\beginpicture
\setcoordinatesystem units <.5cm,.5cm>         
\setplotarea x from -5 to -4, y from -1 to 1    
\plot -4.5 1.25 -4.5 -1.25 /
\plot -4.25 1.25 -4.25 -1.25 /
\ellipticalarc axes ratio 1:1 360 degrees from -4.5 1.25 center 
at -4.375 1.25
\put{$*$} at -4.375 1.25  
\ellipticalarc axes ratio 1:1 180 degrees from -4.5 -1.25 center 
at -4.375 -1.25 
\endpicture
\qquad\hbox{and}\qquad 
\beginpicture
\setcoordinatesystem units <.5cm,.5cm>         
\setplotarea x from -4 to 1.5, y from -2 to 2    
\plot -1.5 2.2 -1.5 1.12 /
\plot -1.5 0.88 -1.5 -0.88 /
\plot -1.5 -1.12 -1.5 -1.75 /
\plot -1.25 2.2 -1.25 1.12 /
\plot -1.25 0.88 -1.25 -0.88 /
\plot -1.25 -1.12 -1.25 -1.75 /
\ellipticalarc axes ratio 1:1 360 degrees from -1.5 2.2 center 
at -1.375 2.2
\put{$*$} at -1.375 2.2  
\ellipticalarc axes ratio 1:1 180 degrees from -1.5 -1.75 center 
at -1.375 -1.75 
\plot  1 2   1 0.5 /
\plot  1 -0.5   1 -1.5 /
\setlinear
\plot -0.3 1.5  -1.1 1.5 /
\ellipticalarc axes ratio 2:1 180 degrees from -1.65 1.5  center 
at -1.65 1.25 
\plot -1.65 1  -0.3 1 /
\plot -0.3 -0.5  -1.1 -0.5 /
\ellipticalarc axes ratio 2:1 180 degrees from -1.65 -.5  center 
at -1.65 -.75 
\plot -1.65 -1  -0.3 -1 /
\setquadratic
\plot  -0.3 1  -0.05 0.8  -0 0.5 /
\plot   0 0.5   0.15 0.2   0.7 0 /
\plot   1 -0.5   0.9 -0.15   0.7 0 /
\plot  -0.3 1.5  -0.05 1.7  -0 2 /
\plot   -0.3 -0.5     -.1 -.425  -0.05 -0.325 /
\plot   -0.05 -0.325   0.15 -0.1   0.35 -0.05 /
\plot   0.65 0.15   0.9 0.25   1 0.5 /
\plot  -0.3 -1  -0.05 -1.2  -0 -1.5 /
\ellipticalarc axes ratio 1:1 180 degrees from 1 2 center 
at 0.5 2
\endpicture
=
~z^{-1}\cdot
\beginpicture
\setcoordinatesystem units <.5cm,.5cm>         
\setplotarea x from -5 to -1.5, y from -1 to 1    
\plot -4.5 1.25 -4.5 -1.25 /
\plot -4.25 1.25 -4.25 -1.25 /
\ellipticalarc axes ratio 1:1 360 degrees from -4.5 1.25 center 
at -4.375 1.25
\put{$*$} at -4.375 1.25  
\ellipticalarc axes ratio 1:1 180 degrees from -4.5 -1.25 center 
at -4.375 -1.25 
\plot -2.7 -0.35    -2.3 -0.35   /
\plot -2 -0.75   -2 -1.25 /
\plot -3 -0.75   -3 -1.25 /
\setquadratic
\plot  -2.3 -0.35  -2.05 -0.45  -2 -0.75 /
\plot  -3 -0.75  -2.95 -0.45  -2.7 -0.35 /
\endpicture
\end{equation}
\begin{equation}
\beginpicture
\setcoordinatesystem units <.5cm,.5cm>         
\setplotarea x from 0 to 1, y from -1 to 1    
\plot 0.3 0.5  0.7 0.5 /
\plot 0.3 -0.5    0.7 -0.5   /
\setquadratic
\plot  0.7 -0.5  0.95 -0.3  1 0 /
\plot  0 0  0.05 -0.3  0.3 -.5 /
\plot  0.7 0.5  0.95 0.3  1 0 /
\plot  0 0  0.05 0.3  0.3 0.5 /
\endpicture
~~=~~
\frac{z-z^{-1}}{q-q^{-1}} + 1 = Z_0^{(0)}.
\end{equation}

\begin{thm} \label{AffineBMWbasis}  Let $W_k$ be the affine BMW algebra and let 
$W_{r,k}(b_1,\ldots, b_r)$ be the cyclotomic BMW algebra as defined in 
Section \ref{sec:W_k-defn}.  Let $d \in D_k$ be a Brauer diagram, where $D_k$ is as in \eqref{Brauerbasis}.
Choose a minimal length expression of $d$ as a product 
of $e_1,\ldots, e_{k-1}, s_1, \ldots, s_{k-1}$,
$$d = a_1\cdots a_\ell, \qquad
a_i\in \{ e_1, \ldots, e_{k-1}, s_1, \ldots, s_{k-1}\},$$
such that the number of $s_i$ in this product is the number of crossings in $d$.  For each
$a_i$ which is in $\{ s_1, \ldots, s_{k-1}\}$ fix a choice of sign $\epsilon_j = \pm 1$ and set
$$T_d = A_1\cdots A_\ell, \qquad\hbox{where}\quad
A_j=\begin{cases}
E_i, &\hbox{if $a_j=e_i$,}\\
T_i^{\epsilon_j}, &\hbox{if $a_j = s_i$.}
\end{cases}
$$	     
For $n_1,\ldots, n_k\in \ZZ$ let
$$T_d^{n_1,\ldots, n_k} 
= Y_{i_1}^{n_1}\cdots Y_{i_\ell}^{n_\ell} T_d Y_{i_{\ell+1}}^{n_{\ell+1}}\cdots Y_{i_k}^{n_k},$$
where, in the lexicographic ordering of the edges $(i_1,j_1), \ldots, (i_k,j_k)$ of $d$, 
$i_1,\ldots, i_\ell$ are in the top row of $d$ and $i_{\ell+1}, \ldots, i_k$
are in the bottom row of $d$.
\item{(a)} If
\begin{align}
\left(
Z_{0}^- - \frac{z}{q-q^{-1}} - \frac{u^2}{u^2-1}
\right)
\left(
Z_{0}^+ + \frac{z^{-1}}{q-q^{-1}} - \frac{u^2}{u^2-1}
\right)
= \frac{-(u^2-q^2)(u^2-q^{-2})}{(u^2-1)^2(q-q^{-1})^2}
\label{Admissibility}
\end{align}
then
$\{T_d^{n_1,\ldots, n_k}\ |\ d\in D_k,\  n_1,\ldots, n_k\in \ZZ \}$
is a $C$-basis of $W_k$.
\item{(b)} If \eqref{Admissibility} holds and 
\begin{equation}\label{CycAdmissibility}
Z_0^+ + \frac{z^{-1}}{q-q^{-1}}-\frac{u^2}{u^2-1}
=
\left( \frac{z}{q-q^{-1}}+\frac{uz}{u^2-1}\right)
\prod_{j=1}^r \frac{u-b_j^{-1}}{u-b_j}
\end{equation}
then
$\{T_d^{n_1,\ldots, n_k}\ |\ d\in D_k,\  0\le n_1,\ldots, n_k\le r-1 \}$
is a $C$-basis of $W_{r,k}(b_1, \dots, b_r)$.
\end{thm}

Part (a) of Theorem \ref{AffineBMWbasis} is \cite[Theorem 2.25]{GH2} and 
part (b) is \cite[Theorem 5.5]{GH2} and \cite[Theorem 8.1]{WY2}. We refer to these references for proof, remarking only that 
one key point in showing that 
$\{T_d^{n_1,\ldots, n_k}\ |\ d\in D_k,\  n_1,\ldots, n_k\in \ZZ \}$ 
spans $W_k$ is that
if $(i,j)$ is a top-to-bottom edge in $d$ then
\begin{equation}\label{vert}
Y_iT_d = T_d Y_j + (\hbox{terms with fewer crossings}),
\end{equation}
and, if $(i,j)$ is a top-to-top edge in $d$ then
\begin{equation}\label{Horiz}
Y_iT_d = Y_j^{-1}T_d + (\hbox{terms with fewer crossings}).
\end{equation}
As an example, let $d = s_1e_3s_5e_2e_4e_1s_3s_5$ and choose $\epsilon_{1} =\epsilon_{3} =-\epsilon_{7} =\epsilon_{8} = 1$. Then
$$d=
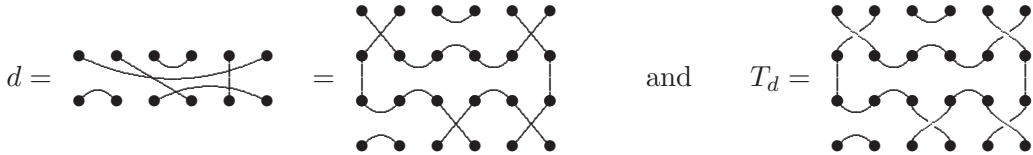
\beginpicture
\setcoordinatesystem units <0.5cm,0.2cm> 
\setplotarea x from 0.5 to 7, y from 0 to 3    
\linethickness=0.5pt                        
\put{$\bullet$} at 1 -1 \put{$\bullet$} at 1 2
\put{$\bullet$} at 2 -1 \put{$\bullet$} at 2 2
\put{$\bullet$} at 3 -1 \put{$\bullet$} at 3 2
\put{$\bullet$} at 4 -1 \put{$\bullet$} at 4 2
\put{$\bullet$} at 5 -1 \put{$\bullet$} at 5 2
\put{$\bullet$} at 6 -1 \put{$\bullet$} at 6 2
\plot 2 2 4 -1 /
\plot 5 2 5 -1 /
\setquadratic
\plot 1 2 3.5 0.5 6 2 /
\plot 3 2 3.5 1.25 4 2 /
\plot 1 -1 1.5 -0.25 2 -1 /
\plot 3 -1 4.5 -0 6 -1 /
\endpicture
=
\beginpicture
\setcoordinatesystem units <0.5cm,0.2cm> 
\setplotarea x from 0.5 to 7, y from 0 to 3    
\linethickness=0.5pt                        
\put{$\bullet$} at 1 -1 \put{$\bullet$} at 1 2
\put{$\bullet$} at 2 -1 \put{$\bullet$} at 2 2
\put{$\bullet$} at 3 -1 \put{$\bullet$} at 3 2
\put{$\bullet$} at 4 -1 \put{$\bullet$} at 4 2
\put{$\bullet$} at 5 -1 \put{$\bullet$} at 5 2
\put{$\bullet$} at 6 -1 \put{$\bullet$} at 6 2
\put{$\bullet$} at 1 2 \put{$\bullet$} at 1 5
\put{$\bullet$} at 2 2 \put{$\bullet$} at 2 5
\put{$\bullet$} at 3 2 \put{$\bullet$} at 3 5
\put{$\bullet$} at 4 2 \put{$\bullet$} at 4 5
\put{$\bullet$} at 5 2 \put{$\bullet$} at 5 5
\put{$\bullet$} at 6 2 \put{$\bullet$} at 6 5
\put{$\bullet$} at 1 -4 \put{$\bullet$} at 1 -1
\put{$\bullet$} at 2 -4 \put{$\bullet$} at 2 -1
\put{$\bullet$} at 3 -4 \put{$\bullet$} at 3 -1
\put{$\bullet$} at 4 -4 \put{$\bullet$} at 4 -1
\put{$\bullet$} at 5 -4 \put{$\bullet$} at 5 -1
\put{$\bullet$} at 6 -4 \put{$\bullet$} at 6 -1
\plot 1 5 2 2 /
\plot 2 5 1 2 /
\plot 5 5 6 2 /
\plot 6 5 5 2 /
\plot 3 -1 4 -4 /
\plot 4 -1 3 -4 /
\plot 5 -1 6 -4 /
\plot 6 -1 5 -4 /
\plot 1 2 1 -1 /
\plot 6 2 6 -1 /
\setquadratic
\plot 3 5 3.5 4.25 4 5 /
\plot 3 2 3.5 2.75 4 2 /
\plot 2 2 2.5 1.25 3 2 /
\plot 2 -1 2.5 -0.25 3 -1 /
\plot 4 2 4.5 1.25 5 2 /
\plot 4 -1 4.5 -0.25 5 -1 /
\plot 1 -1 1.5 -1.75 2 -1 /
\plot 1 -4 1.5 -3.25 2 -4 /
\endpicture
\qquad\hbox{and}\qquad T_d = 
\beginpicture
\setcoordinatesystem units <0.5cm,0.2cm> 
\setplotarea x from 0.5 to 7, y from 0 to 3    
\linethickness=0.5pt                        
\put{$\bullet$} at 1 -1 \put{$\bullet$} at 1 2
\put{$\bullet$} at 2 -1 \put{$\bullet$} at 2 2
\put{$\bullet$} at 3 -1 \put{$\bullet$} at 3 2
\put{$\bullet$} at 4 -1 \put{$\bullet$} at 4 2
\put{$\bullet$} at 5 -1 \put{$\bullet$} at 5 2
\put{$\bullet$} at 6 -1 \put{$\bullet$} at 6 2
\put{$\bullet$} at 1 2 \put{$\bullet$} at 1 5
\put{$\bullet$} at 2 2 \put{$\bullet$} at 2 5
\put{$\bullet$} at 3 2 \put{$\bullet$} at 3 5
\put{$\bullet$} at 4 2 \put{$\bullet$} at 4 5
\put{$\bullet$} at 5 2 \put{$\bullet$} at 5 5
\put{$\bullet$} at 6 2 \put{$\bullet$} at 6 5
\put{$\bullet$} at 1 -4 \put{$\bullet$} at 1 -1
\put{$\bullet$} at 2 -4 \put{$\bullet$} at 2 -1
\put{$\bullet$} at 3 -4 \put{$\bullet$} at 3 -1
\put{$\bullet$} at 4 -4 \put{$\bullet$} at 4 -1
\put{$\bullet$} at 5 -4 \put{$\bullet$} at 5 -1
\put{$\bullet$} at 6 -4 \put{$\bullet$} at 6 -1
\setlinear
\plot 1 2 1 -1 /
\plot 6 2 6 -1 /
\setquadratic
\plot 3 5 3.5 4.25 4 5 /
\plot 3 2 3.5 2.75 4 2 /
\plot 2 2 2.5 1.25 3 2 /
\plot 2 -1 2.5 -0.25 3 -1 /
\plot 4 2 4.5 1.25 5 2 /
\plot 4 -1 4.5 -0.25 5 -1 /
\plot 1 -1 1.5 -1.75 2 -1 /
\plot 1 -4 1.5 -3.25 2 -4 /
\plot  1 2  1.05 2.6  1.4 3.3 /
\plot  1.6 3.7  1.95 4.4  2 5 /
\plot 1 5  1.05 4.4  1.5 3.5  1.95 2.6  2 2 /
\plot  5 2  5.05 2.6  5.4 3.3 /
\plot  5.6 3.7  5.95 4.4  6 5 /
\plot 5 5  5.05 4.4  5.5 3.5  5.95 2.6  6 2 /
\plot  5 -4  5.05 -3.4  5.4 -2.7 /
\plot  5.6 -2.3  5.95 -1.6  6 -1 /
\plot 5 -1  5.05 -1.6  5.5 -2.5  5.95 -3.4  6 -4 /
\plot  4 -4  3.95 -3.4  3.6 -2.7 /
\plot  3.4 -2.3  3.05 -1.6  3 -1 /
\plot 4 -1  3.95 -1.6  3.5 -2.5  3.05 -3.4  3 -4 /
\endpicture
$$
so that $T_d = T_1E_3T_5E_2E_4E_1T_3^{-1}T_5$ and
$T_d^{5,3,-2,0,3,0} = Y_1^5 Y_2^3 Y_3^{-2}T_d Y_1^3$.
Then, since $(1,6)$ is a horizontal edge in $d$, \eqref{Horiz} is illustrated by the computation
\begin{align}
Y_6T_d 
&= T_1E_3Y_6T_5E_2E_4E_1T_3^{-1}T_5
= T_1E_3(T_5Y_5+(q-q^{-1})Y_6(1-E_5))E_2E_4E_1T_3^{-1}T_5 
\nonumber \\
&= T_1E_3T_5E_2Y_5E_4E_1T_3^{-1}T_5 + \cdots 
= T_1E_3T_5E_2Y_4^{-1}E_4E_1T_3^{-1}T_5 + \cdots 
\nonumber \\
&= T_1E_3Y_4^{-1}T_5E_2E_4E_1T_3^{-1}T_5 + \cdots 
= T_1E_3Y_3T_5E_2E_4E_1T_3^{-1}T_5 + \cdots 
\label{basisexample} \\
&= T_1E_3T_5Y_3E_2E_4E_1T_3^{-1}T_5 + \cdots 
= T_1E_3T_5Y_2^{-1}E_2E_4E_1T_3^{-1}T_5 + \cdots 
\nonumber \\
&= T_1Y_2^{-1}E_3T_5E_2E_4E_1T_3^{-1}T_5 + \cdots 
= Y_1^{-1}T_1E_3T_5E_2E_4E_1T_3^{-1}T_5 + \cdots, 
\nonumber
\end{align}
where $+\cdots$ is always a linear combination of terms with fewer crossings. 


\subsection{The center of $W_k$}

The affine BMW algebra is the algebra $W_k$ over $C$ defined in Section \ref{sec:W_k-defn} and the ring of Laurent polynomials 
$C[Y_1^{\pm1},\ldots, Y_k^{\pm1}]$ is a subalgebra of $W_k$ (see Remark \ref{Polysubalg}). 
The symmetric group $S_k$ acts on $C[Y_1^{\pm1},\ldots, Y_k^{\pm1}]$ by permuting the variables and the ring of symmetric functions is
$$C[Y_1^{\pm1},\ldots,Y_k^{\pm1}]^{S_k}
=\{ f\in C[Y_1^{\pm1},\ldots, Y_k^{\pm1}]\ |\ \hbox{$wf=f$, for $w\in S_k$}\}.
$$
A classical fact (see, for example, \cite[Proposition 2.1]{GV}) 
is that the center of the affine Hecke algebra $H_k$ is 
$$Z(H_k) = C[Y_1^{\pm1},\ldots,Y_k^{\pm1}]^{S_k}.$$
Theorem \ref{thm:BWMcenter}  is a characterization of the center of the
affine BMW algebra.

\begin{thm}\label{thm:BWMcenter} The center of the affine BMW algebra $W_k$ is 
$$R_k =  \{ f \in C[Y_1^{\pm 1},\ldots, Y_k^{\pm 1}]^{S_k}\ |\ 
f(Y_1,Y_1^{-1}, Y_3, \ldots, Y_k) =  f(1,1, Y_3,  \ldots, Y_k) \}.$$
\end{thm}

\begin{proof}
\emph{Step 1: $f\in W_k$ commutes with all $Y_i$ $\Leftrightarrow$ $f \in C[Y_1^{\pm1}, \dots, Y_k^{\pm1}]$}:\\
Assume $f \in W_k$ and write
$$f = \sum c_d^{n_1,\ldots, n_k} T_d^{n_1,\ldots, n_k},$$
in terms of the basis in Theorem \ref{AffineBMWbasis}. 
Let $d\in D_k$ with the maximal number of crossings
such that $c_d^{n_1,\ldots, n_k}\ne 0$ and, using the notation before \eqref{Brauerbasis},
suppose there is an edge $(i,j)$ of $d$ such that $j\ne i'$.  Then,
by \eqref{vert} and \eqref{Horiz},  
$$\hbox{the coefficient of}\quad Y_iT_d^{n_1,\ldots, n_k}\quad\hbox{in}\quad
Y_i f \quad\hbox{is}\quad c_d^{n_1,\ldots, n_k}$$
and
$$\hbox{the coefficient of}\quad Y_iT_d^{n_1,\ldots, n_k}\quad\hbox{in}\quad
f Y_i \quad\hbox{is}\quad 0.$$
If $Y_i f= f Y_i$ it follows that there is no such edge, and so $d=1$ (and therefore $T_d=1$). Thus $f \in C[Y_1^{\pm1},\ldots, Y_k^{\pm1}]$. Conversely, if $f \in C[Y_1^{\pm1},\ldots, Y_k^{\pm1}]$, then $Y_i f= f Y_i$. 

\noindent 
\emph{Step 2: $f\in C[Y_1^{\pm1},\ldots, Y_k^{\pm1}]$ commutes with all 
$T_i$ $\Leftrightarrow$ $f \in R_k$}:\\
Assume $f \in C[Y_1^{\pm1}, \dots, Y_k^{\pm1}]$ and write 
\begin{equation*}
f = \sum_{a,b \in \ZZ } Y_1^a Y_2^b f_{a,b}, \qquad \text{ where } f_{a,b} \in C[Y_3^{\pm1}, \dots, Y_k^{\pm1}].
\end{equation*}
Then $\displaystyle{
 f(1, 1, Y_3, \dots, Y_k) = \sum_{a,b \in \ZZ} f_{a,b}
}$ and
\begin{equation}\label{YYinv}
f(Y_1, Y_1^{-1}, Y_3, \dots, Y_k) = \sum_{a,b \in \ZZ} Y_1^{a-b} f_{a,b}
= \sum_{\ell \in \ZZ} Y_1^\ell \left( \sum_{b\in \ZZ} f_{\ell+b, b}\right).
\end{equation}
By direct computation using \eqref{eq:TYip1} and \eqref{eq:TYip1neg},
\begin{align*}
T_1 Y_1^a Y_2^b 
&= Y_1^a Y_2^a T_1 Y_2^{b-a}
= s_1(Y_1^aY_2^b)  T_1+ (q-q^{-1})\frac{Y_1^aY_2^b - s_1(Y_1^aY_2^{b})}{1-Y_1Y_2^{-1}} 
+ \cE_{b-a},
\end{align*}
where
$$
\cE_\ell = \begin{cases} 
\displaystyle{-(q-q^{-1}) \sum_{r= 1}^{\ell}  Y_1^{\ell - r}E_1Y_1^{-r},}& \hbox{if $\ell>0$,}\\
\displaystyle{(q-q^{-1}) \sum_{r= 1}^{-\ell}  Y_1^{\ell+r-1}E_1Y_1^{r-1},
} &\hbox{if $\ell<0$,}\\
0, &\hbox{if $\ell = 0$.}
\end{cases}
$$ 
It follows that
\begin{equation}
\label{T1pastf}
T_1 f =  (s_1f) T_1 + (q-q^{-1}) \frac{f - s_1 f }{1 - Y_1Y_2^{-1}} 
+ \sum_{\ell \in \ZZ_{\neq 0}} \cE_\ell  \left(\sum_{b\in\ZZ}f_{\ell+b,b}\right).
\end{equation}
Thus, if $f(Y_1, Y_1^{-1}, Y_3, \dots, Y_k) = f(1, 1, Y_3, \dots, Y_k)$ then, by
\eqref{YYinv},
\begin{equation}\label{E-sym-cancellation}
 \sum_{b \in \ZZ} f_{\ell+b, b} = 0, \qquad \hbox{for $\ell \neq 0$.}
 \end{equation}
Hence, if $f\in C[Y_1^{\pm1}, \dots, Y_k^{\pm1}]^{S_k}$ 
and $f(Y_1, Y_1^{-1}, Y_3, \dots, Y_k) = f(1, 1, Y_3, \dots, Y_k)$ then
$s_1f=f$ and \eqref{E-sym-cancellation} holds so that, by \eqref{T1pastf},
$T_1f=fT_1$.  Similarly, $f$ commutes with all $T_i$.

\noindent
Conversely, if $f\in C[Y_1^{\pm1}, \ldots, Y_k^{\pm1}]$ and $T_if=fT_i$ then
$$s_if=f
\qquad\hbox{and}\qquad
 \sum_{b \in \ZZ} f_{\ell+b, b} = 0, \qquad \hbox{for $\ell \neq 0$,}
$$
so that $f \in C[Y_1^{\pm 1},\ldots, Y_k^{\pm 1}]^{S_k}$ and
$f(Y_1,Y_1^{-1}, Y_3, \ldots, Y_k) =  f(1,1, Y_3,  \ldots, Y_k)$.

It follows from \eqref{rel:E_Defn} that $R_k = Z(W_k)$. 
\end{proof}

The symmetric group $S_k$ acts on $\ZZ^k$ by permuting the factors.  The ring 
$$C[Y_1^{\pm1},\ldots, Y_k^{\pm1}]^{S_k}
\qquad\hbox{has basis}\qquad
\{ m_\lambda\ |\ \hbox{$\lambda\in \ZZ^k$ 
with $\lambda_1\ge \lambda_2\ge \cdots \ge \lambda_k$}\},$$
where
$$m_\lambda = \sum_{\mu\in S_k\lambda} Y_1^{\mu_1}\cdots Y_k^{\mu_k}.
$$
The \emph{elementary symmetric functions}
are
$$e_r = m_{(1^r, 0^{k-r})}
\quad\hbox{and}\quad
e_{-r} = m_{(0^{k-r}, (-1)^r)},\qquad\hbox{for $r=0,1,\ldots, k$},$$
and the \emph{power sum symmetric functions} are  
$$p_r = m_{(r, 0^{k-1})}
\quad\hbox{and}\quad
p_{-r} = m_{(0^{k-1}, -r)},\qquad\hbox{for $r\in \ZZ_{>0}$.}$$
The Newton identities (see \cite[Ch.\ I ($2.11'$)]{Mac})  say
\begin{equation}\label{newtonids}
\ell e_\ell = \sum_{r=1}^\ell (-1)^{r-1} p_r e_{\ell-r}
\qquad\hbox{and}\qquad
\ell e_{-\ell} = \sum_{r=1}^{\ell} (-1)^{r-1} p_{-r} e_{-(\ell-r)},
\end{equation}
where the second equation is obtained from the first by replacing $Y_i$ with $Y_i^{-1}$.
For $\ell\in \ZZ$ and $\lambda = (\lambda_1,\ldots, \lambda_k)\in \ZZ^k$,   
$$e_k^\ell m_\lambda = m_{\lambda+(\ell^k)},
\qquad\hbox{where}\quad
\lambda+(\ell^k) = (\lambda_1+\ell,\ldots, \lambda_k+\ell).
$$
In particular,
\begin{equation}\label{varswitches}
e_{-r} = e_k^{-1}e_{k-r},\qquad\hbox{for $r=0,\ldots, k$.}
\end{equation}
Define
\begin{equation}
p_i^+ = p_i + p_{-i}
\qquad\hbox{and}\qquad
p_i^- = p_i - p_{-i},
\quad\hbox{for $i\in \ZZ_{>0}$.}
\end{equation}
The consequence of \eqref{varswitches} and \eqref{newtonids} is that 
\begin{align*}
\CC[Y_1^{\pm 1}, \dots, Y_k^{\pm 1}]^{S_k}
&= \CC[e_k^{\pm1}, e_1,\dots , e_{k-1}] \\
&= \CC[e_k^{\pm1}][e_1,e_2,\ldots, e_{\lfloor \frac{k}{2}\rfloor},
e_ke_{-\lfloor \frac{k-1}{2}\rfloor}, \dots , e_ke_{-2}, e_ke_{-1}] \\
&= \CC[e_k^{\pm1}][e_1,e_2,\ldots, e_{\lfloor \frac{k}{2}\rfloor},
e_{-\lfloor \frac{k-1}{2}\rfloor}, \dots , e_{-2}, e_{-1}] \\
&= \CC[e_k^{\pm1}][p_1,p_2,\ldots, p_{\lfloor \frac{k}{2} \rfloor},
p_{-\lfloor \frac{k-1}{2}\rfloor}, \dots , p_{-2}, p_{-1}] \\
&= \CC[e_k^{\pm1}][p^+_1,p^+_2,\ldots, p^+_{\lfloor \frac{k}{2}\rfloor},
p^-_{\lfloor \frac{k-1}{2}\rfloor}, \dots , p^-_{2}, p^-_{1}].
\end{align*}
For $\nu\in \ZZ^k$ with $\nu_1\ge \cdots \ge \nu_\ell >0$ define
$$p^+_\nu = p^+_{\nu_1}\cdots p^+_{\nu_\ell}
\qquad\hbox{and}\qquad
p^-_\nu = p^-_{\nu_1}\cdots p^-_{\nu_\ell}.
$$
Then
\begin{equation}
\CC[Y_1^{\pm 1}, \dots, Y_k^{\pm 1}]^{S_k}
\qquad\hbox{has basis}\qquad
\{ e_k^\ell p^+_\lambda p^-_\mu\ |\ \ell\in \ZZ, \ell(\lambda)\le \hbox{$\lfloor \frac{k}{2}\rfloor$},
\ell(\mu) \le \hbox{$\lfloor \frac{k-1}{2}\rfloor$}\}.
\end{equation}
In analogy with \eqref{oddpower} we expect that if $R_k$ is as in Theorem 
\ref{thm:BWMcenter} then
\begin{equation}
R_k = C[e_k^{\pm1}][p_1^-,p_2^-,\ldots].
\end{equation}

\begin{remark}  The left ideal of $W_2$ generated by $E_1$ is 
$C[Y_1^{\pm1}]E_1$. This is an infinite dimensional (generically irreducible) $W_2$-module on which $Z(W_2)$ acts by constants. It follows that
$W_2$ is not a finitely generated $Z(W_2)$-module.
\end{remark}

\end{document}